\documentclass[preprint,12pt]{elsarticle}

\usepackage{amssymb}
\usepackage{amsmath}
\usepackage{mathtools}
\usepackage{amsfonts}
\usepackage{amsthm}
\usepackage{float}
\usepackage{times}
\usepackage{tikz-cd}
\usepackage{mathabx}
\usepackage{stmaryrd}
\usepackage{multicol}
\usetikzlibrary{arrows}

\floatstyle{boxed}
\newtheorem{theorem}{Theorem}[section]

\newtheorem{lemma}[theorem]{Lemma}
\newtheorem{corollary}[theorem]{Corollary}
\theoremstyle{definition}
\newtheorem{example}[theorem]{Example}

\newcommand{\m}{\mathbf} 

\newcommand{\ra}{\mathbin{\rightarrow}}
\newcommand{\jn}{\vee}
\newcommand{\mt}{\wedge}

\newcommand{\ls}{\setbox0\hbox{$-$}
\mathbin{\hbox{$-$\kern-\wd0\raise2\dp0\hbox{$\cdot$}\kern.3\wd0\lower2\dp0\hbox{$\cdot$}}}}
\newcommand{\rs}{\setbox0\hbox{$-$}
\mathbin{\hbox{$-$\kern-\wd0\lower2\dp0\hbox{$\cdot$}\kern.3\wd0 \raise2\dp0\hbox{$\cdot$}}}}

\newcommand{\lsc}{\setbox0\hbox{$-$}
\mathbin{\hbox{$-$\kern-\wd0\raise2\dp0\hbox{$\cdot$}\kern.3\wd0\lower2\dp0\hbox{\phantom{$\cdot$}}}}}

\newcommand{\rsc}{\setbox0\hbox{$-$}
\mathop{\hbox{$-$\kern-\wd0\lower2\dp0\hbox{\phantom{$\cdot$}}\kern.3\wd0\raise2\dp0\hbox{$\cdot$}}}}

\newcommand{\lrsc}{\setbox0\hbox{$-$}
\mathop{\hbox{$-$\kern-\wd0\raise2\dp0\hbox{$\cdot$}\kern.3\wd0\raise2\dp0\hbox{$\cdot$}}}}

\renewcommand{\rm}{\cdot} 
\newcommand{\diagcov}{\Yleft}
\newcommand{\preorder}{\trianglelefteq}

\journal{Journal of Algebra}

\begin{document}

\begin{frontmatter}

\title{Decidability of distributive $\ell$-pregroups}
\author[inst1]{Nikolaos Galatos}
\ead{ngalatos@du.edu}

\affiliation[inst1]{organization= {Department of Mathematics, University of Denver},
            addressline={2360 S. Gaylord St.}, 
            city={Denver},
            postcode={80208}, 
            state={CO},
            country={USA}}

\author[inst1]{Isis A. Gallardo}
\ead{isis.gallardo@du.edu}

\begin{abstract}
 We show that every distributive lattice-ordered pregroup can be embedded into a functional algebra over an integral chain, thus improving the existing  Cayley/Hol\-land-style embedding theorem. We use this to show that the variety of all distributive lattice-ordered pregroups is generated by the single functional algebra on the integers. Finally, we show that the equational theory of the variety is decidable. 
\end{abstract}




\end{frontmatter}


\section{Introduction}
A \emph{pregroup} $(A,\cdot,1,^{\ell} ,^{r},\leq)$ is a partially
ordered monoid $(A,\cdot,1,\leq)$, where multiplication preserves the order and
$$
x^{\ell} x\leq1\leq xx^{\ell} \text{ and }xx^{r}\leq1\leq x^{r}x.
$$
Pregroups are of importance in mathematical linguistics; see \cite{La}, \cite{Bu}, \cite{Ba}, for example. 
A \emph{lattice-ordered pregroup} (\emph{$\ell$-pregroup}) is an algebra  $(A,\wedge,\vee,\cdot,1,^{\ell} ,^{r})$,
such that $(A,\wedge,\vee)$ is a lattice and $(A,\cdot,1,^{\ell} ,^{r},\leq)$ is a pregroup, where $\leq$ is the lattice order.
Moreover, $\ell$-pregroups are precisely the involutive residuated lattices for which multiplication is the same as its De Morgan dual (known as addition); see \cite{GJKO} for a study of residuated lattices and their applications to substructural logics. Therefore, $\ell$-pregroups can be compared to other residuated structures such as Boolean algebras, relation algebras, the residuated lattice of ideals of a ring with unit, Heyting algebras, and MV-algebras, to name a few. A \emph{distributive} $\ell$-pregroup is an $\ell$-pregroup where the lattice
reduct is distributive. In an $\ell$-pregroup the demand that multiplication preserves the order is actually equivalent to asking that multiplication distributes over join/meet, so the class of all distributive $\ell$-pregroups
forms a variety, which we denote by $\mathsf{DLPG}$. 

 Prominent examples of distributive $\ell$-pregroups are lattice-ordered groups ($\ell$-groups): algebras $(G, \mt, \jn, \cdot, ^{-1}, 1)$ with a group $(G, \cdot, ^{-1}, 1)$ and with a lattice $(G, \mt, \jn)$ reduct where multiplication is order-preserving (equivalently it distributes over join or meet).  
 It turns out that $\ell$-groups are precisely the (distributive) $\ell$-pregroups that satisfy $x^\ell=x^r$; i.e., where the two inverses (right and left) coincide. Lattice distributivity follows from the axioms of $\ell$-groups (due to the existence of a single inverse), but it is still an open problem whether all $\ell$-pregroups are distributive; see \cite{GJ} and \cite{GJKP} for partial results in that direction.  Lattice-ordered groups have been studied extensively (see \cite{AF}, \cite{Da}, \cite{GHo}, \cite{KM}) and they enjoy a Cayley-type theorem: every $\ell$-group can be embedded into a \emph{symmetric} one (\emph{Holland's embedding theorem} \cite{Ho-em}). Here, the \emph{symmetric} $\ell$-group over a totally-ordered set $(X, \leq)$ consists of all permutations of $X$ that are order-preserving; this symmetric $\ell$-group is a subgroup of the symmetric group on $X$ and the order is given pointwise. 
 
  A similar Cayley/Holland-type embedding theorem is also known for distributive $\ell$-pregroups: every distributive $\ell$-pregroup can be embedded into a \emph{symmetric}/\emph{functional} one; see \cite{GH}. For example, the \emph{functional} $\ell$-pregroup $\m F(\mathbb{Z})$ on the totally-ordered set $\mathbb{Z}$ consists of all finite-to-one  (the preimage of every singleton is a finite set) order-preserving functions; this contains the symmetric $\ell$-group of all translations, $x \mapsto x+n$ for various $n \in \mathbb{Z}$, but it is a much bigger algebra. More generally, the \emph{functional} $\ell$-pregroup  $\m F(\m \Omega)$ on a chain $\mathbf{\Omega}$ consists of the (order-preserving) functions on  $\mathbf{\Omega}$  that have \emph{residuals} and \emph{dual residuals} of all orders, under composition and pointwise order. Therefore, distributive $\ell$-pregroups are a very wide generalization of $\ell$-groups, but at the same time they are amenable to some of the same analysis. 

  Furthermore, \emph{Holland's generation theorem} \cite{Ho-gen} shows that the single symmetric $\ell$-group on the chain $\mathbb{R}$ generates the whole variety of  $\ell$-groups; it turns out that this symmetric $\ell$-group is equal to $\m F(\mathbb{R})$. A main result of our paper is a generation theorem for distributive $\ell$-pregroups: the variety $\mathsf{DLPG}$ is generated by $\m F(\mathbb{Z})$. This is quite surprising, as it implies that every equation that fails in $\m F(\mathbb{R})$ also fails in $\m F(\mathbb{Z})$, despite the fact that the chain $\mathbb{Z}$ is countable and discrete; one would expect that a lexicographic product of $\mathbb{R}$ and  $\mathbb{Z}$ would be needed instead. 

  As mentioned above, every distributive $\ell$-pregroup embeds into $\m F(\m \Omega)$
for some chain $\mathbf{\Omega}$. In Section~\ref{s: integral} we improve this result by proving that the chain $\m \Omega$ can be taken to be \emph{integral}: every element of the chain is contained in an interval isomorphic to the chain of the integers. We obtain this theorem by defining an integral extension $\overline{\m \Omega}$ of $\m \Omega$ and constructing an $\ell$-pregroup embedding of $\m F(\m \Omega)$ into $\m F(\overline{\m \Omega})$. Therefore, if an equation fails in $\mathsf{DLPG}$ then it fails in $\m F(\m \Omega)$ for some integral chain $\m \Omega$. This improved embedding theorem requires quite a bit of work, but it serves as the first step toward the generation by $\m F(\mathbb{Z})$.

In Section~\ref{s: diagrams} we show that if an  equation fails in 
   $\m F(\m \Omega)$ for some integral chain $\m \Omega$, then there is a \emph{diagram} where it fails; a diagram consists of a finite chain with designated covering pairs and partial functions on it and is inspired by \cite{HM}. Failure in a diagram entails that the partial functions are defined enough so as to demonstrate the failure of the equation. We also show that there are only finitely many possible diagrams that demonstrate the failure of an equation and describe an algorithm that produces them all. We note that in diagrams for $\ell$-groups the single inverse of a partial function is automatically obtained by inverting the direction of the arrows. However, calculating the two inverses in diagrams for distributive $\ell$-pregroups is much more involved. Within the proofs, this ends up requiring an expansion of the language to include two more constants corresponding to the successor and the predecessor functions on an integral chain, which in turn guide the production of the required set of points of the diagram and the corresponding arrows; the associated inductive argument ends up taking a digression into this expanded language. 

   In Section~\ref{s: F(Z)} we close the cycle by showing that if an equation fails in a diagram, then it also fails in $\m F(\mathbb{Z})$; hence also in $\mathsf{DLPG}$. Together with the results of Section~\ref{s: diagrams} this shows that an equation holds in  $\mathsf{DLPG}$ iff it holds in $\m F(\mathbb{Z})$; therefore $\m F(\mathbb{Z})$ generates the variety $\mathsf{DLPG}$. Actually, we do a bit better by showing that the subalgeba of $\m F(\mathbb{Z})$ that consists of the functions of finite support also generates the whole variety.
   As another consequence of closing the cycle, we have that an equation fails in  $\mathsf{DLPG}$ iff it fails in one of the finitely-many diagrams produced by the algorithm of Section~\ref{s: diagrams}. Therefore, the equational theory of  $\mathsf{DLPG}$ is decidable.

   Some of the results in Sections~\ref{s: diagrams} and~\ref{s: F(Z)} are actually phrased in terms of \emph{compatible surjections}, instead of diagrams, as they allow for tighter control in the formulation of the decidability algorithm.  
 In Section 5 we show that compatible surjections are in bijective correspondence with \emph{compatible preorders}, showing an alternative and equivalent way of viewing failures of equations. Moreover, since compatible preorders for $\ell$-groups are connected to right total orders on (free) groups, see \cite{CM} and \cite{CGMS}, this allows for their comparative study in the case of distributive $\ell$-pregroups. We conclude with some examples of invalid equations, by giving the diagram, the compatible surjection, and the compatible preorder that exhibit the failure. 

\section{Embedding into $\mathbf{F(\Omega)}$, where $\mathbf{\Omega}$ is integral}\label{s: integral}

In this section, we improve on the Cayley/Holland-type embedding theorem of \cite{GH} and show that every distributive $\ell$-pregroup can be embedded in $\mathbf{F(\Omega)}$ for some \emph{integral} chain $\mathbf{\Omega}$, i.e., a chain where every element is contained in an interval isomorphic to $\mathbb{Z}$.

\subsection{Residuals and dual residuals}

Given functions $f:\mathbf{P}\rightarrow\mathbf{Q}$ and $g:\mathbf{Q}\rightarrow\mathbf{P}$ between posets, we say that $g$ is a \emph{residual} for $f$, or that $f$ is a \emph{dual residual} for $g$, or that $(f,g)$ forms a \emph{residuated pair} if 
\[
f(a)\leq b\Leftrightarrow a\leq g(b)\text{, for all }a\in P,b\in Q.
\]

The residual of $f$ is unique when it exists, and we denote it by $f^r$; in this case, we have
\[
f^{r}(b)=\max \{ a\in \Omega :f(a)\leq b \}.
\]
Also, the dual residual of $f$ is unique, when it exists, and we denote it by $f^{\ell}$; in this case, we have
\[
f^{\ell}(a)=\min \{ b\in \Omega :a\leq f (b) \}.
\]
If $f$ has a residual, then $f$ is called \emph{residuated} and if it has a dual residual, it is called \emph{dually residuated}.

It is well known that a function $f:\mathbf{P}\rightarrow\mathbf{Q}$ is residuated iff
it is order-preserving 
and, for all $b \in Q$, the set $\{ a\in P :f(a)\leq b \} $ has a maximum in $\mathbf{Q}$; also, it is dually residuated iff
it is order-preserving 
and, for all $b \in Q$, the set $\{ b\in P :a\leq f (b) \}$ has a minimum. Also, if a function is residuated then it preserves all existing joins, and if it is dually residuated then it preserves all existing meets. We refer to \cite{GJKO} for basics of residuation theory.

As usual, we denote the inverse image of a set $X$ via a function $f$ by $f^{-1}[X]$; if $X$ is a singleton instead of, $f^{-1}[\{a\}]$ we write $f^{-1}[a]$. Also, we write $a \prec b$, when $a$ is covered by $b$ (i.e., when $b$ is a cover of $a$).
The following lemma shows certain special properties of residuation when the posets are chains. 

\begin{lemma}\label{l: facts of fl and fr}
     If $f$ is an order-preserving map on a chain $\mathbf{\Omega}$ with residual $f^r$ and dual residual $f^{\ell}$, and $a,b \in \Omega$, then:
    \begin{enumerate}
            \item $f^{\ell}ff^{\ell}=f^{\ell}$, $f^{r}ff^{r}=f^{r}$, 
        $ff^{\ell}f=f$ and $ff^{r}f=f$.
        \item We have $b<f(a)$ iff $f^r(b)<a$. Also, $f(b)<a$ iff $b<f^{\ell}(a)$.
        \item If  $a\in f[\Omega]$, then $f^{\ell}(a)\leq f^r(a)$ and $f^{-1}[a]=[f^{\ell}(a),f^r(a)]$.
        \item  If   $a\notin f[\Omega]$, then $f^r(a)\prec f^{\ell}(a)$ and $a\in (ff^r(a), ff^{\ell}(a))$.
    \end{enumerate}
\end{lemma}
\begin{proof}
The proof of (1) can be found in \cite{GJKO}.

    (2) Suppose $a,b\in \Omega$. Since $\mathbf{\Omega }$ is a chain $b<f(a)$ iff $f(a)\nleq b$. Using residuation, we can see that $f(a)\nleq b$ iff $a\nleq f^r(b)$ iff $f^r(b)<a$, since $\mathbf{\Omega }$ is a chain. Similarly, we can show that $f(b)<a$ iff $b<f^{\ell}(a)$.

    (3) First note that if $z \in f^{-1}[ a ]$, then  $f(z)=a$, which ensures that $f^{\ell} (a)= \min\{ x\in\Omega:f(x)\leq a\} \leq z$ and $f^{r}(a)=\max\{ x\in\Omega:f(x)\geq a\} \geq z$; thus $f^{\ell}(a)\leq z\leq f^r(a)$. Therefore, $f^{-1}[ a ]\subseteq [f^{\ell} (a),f^{r}(a)]$. In particular, since  $a\in f[\Omega]$, we get that $f^{-1}[ a ]$ is non-empty, hence $[f^{\ell} (a),f^{r}(a)]$ is non-empty and    
     $f^{\ell}(a)\leq f^r(a)$. Moreover, let $b$ be an element of $\Omega$ such that $f(b)=a$. Then, by (1) we  get $f(f^r(a))= ff^rf(b)= f(b)=a$ and $f(f^{\ell}(a))= f f^\ell f(b)= f(b)=a$. So, if $f^\ell(a) \leq x \leq f^r(a)$, then 
$a=ff^\ell(a) \leq f(x) \leq ff^r(a)=a$, so $x \in f^{-1}[ a ]$; hence $[f^{\ell} (a),f^{r}(a)]\subseteq f^{-1}[ a ]$. Thus, $f^{-1}[ a ]= [f^{\ell} (a),f^{r}(a)]$.
  
    (4) If $a \not \in f[\Omega]$, then $f^r(a)=\max \{ b\in \Omega :f(b)\leq a \}$ yields $f(f^r(a))<a$. Similarly we get $a<f(f^{\ell}(a))$, so we obtain $a\in (ff^{r}(a),ff^{\ell}(a))$.

    In particular, we cannot have $f^\ell(a)\leq f^r(a)$, because the order preservation of $f$ would give $ff^\ell(a)\leq ff^r(a)$, a contradiction; thus $f^r(a)< f^{\ell}(a)$, since $\Omega$ is a chain.
    Furthermore, if $f^r(a) \leq z \leq f^{\ell}(a)$, then   by residuation $a\leq f(z)\leq a$, so $f(z)=a$, a contradiction. Hence, $f^r(a)\prec f^{\ell}(a)$.
\end{proof}

We now provide a characterization for maps on a chain that are residuated and dually residuated.

\begin{lemma} \label{l: bounded preimage}
Given a chain $\mathbf{\Omega}$, a map $f$  on $\Omega$ is residuated and dually residuated iff $f$ is order-preserving and for all $a\in \Omega$: 
\begin{enumerate}
    \item If $a\in f[\Omega]$, then $f^{-1}[ a]=[b,c]$, for some $b \leq c$ in $\Omega$.
   \item  If $a\notin f[\Omega]$, there exists $b,c\in\Omega$ such that, $b\prec c$ and $ \ensuremath{a\in(f(b),f(c))}$.
\end{enumerate}
In the case (1), $f^\ell(a)=b$ and $f^r(a)=c$; in the case (2), $f^\ell(a)=c$ and $f^r(a)=b$.
\end{lemma}
\begin{proof} The forward direction follows from Lemma~\ref{l: facts of fl and fr}(3,4). For the reverse direction, we show that $\min \{ z\in\Omega:\, a \leq f(z)\}$ and $\max \{ z\in\Omega:\,f(z) \leq a\}$ exist, for all $a \in \Omega$.

If  $a\in f[\Omega]$, by (1) we have $c=\max \{ z\in\Omega:\,f(z)= a\}$. (We used the fact that $b \leq c$, so $[b,c]$ is not empty and $c \in [b,c]$.) Note that for all $z \in \Omega$ with $f(z) \leq a$, we have $z \leq c$: otherwise $c < z$ (since $\Omega$ is a chain) and $a=f(c) \leq f(z)\leq a $ (since  $f$ is order-preserving), violating that $c=\max \{ z\in\Omega:\,f(z)= a\}$. So, we get  $c=\max \{ z\in\Omega:\,f(z) \leq a\}$. Likewise, $b = \min \{ z\in\Omega:\, a \leq f(z)\}$.

 If $a\notin f[\Omega]$, then by (2) we have $a\in(f(b),f(c))$, therefore $f(b) < a$ and $b \in \{ z\in\Omega:\,f(z)\leq a\}$. Also, for all $z \in \Omega$ with $f(z) \leq a$, we have $z \leq b$:  otherwise $b < z$ (since $\Omega$ is a chain), so $c \leq z$ (since $b \prec c$ and $\Omega$ is a chain) and $a < f(c) \leq f(z)\leq a$ (since  $f$ is order-preserving), a contradiction. Therefore, $b = \max\{ z\in\Omega:\,f(z)\leq a\}$. Likewise,  $c=\min \{ z\in\Omega:\,a \leq f(z)\}$.
\end{proof}

If $f$ is residuated with residual $f^r$ and $f^r$ is residuated with residual $f^{rr}:=(f^r)^r$, then we say that $f^{rr}$ is the \emph{second-order} residual of $f$; we also write $f^{r^2}$ for $f^{rr}$. More generally, $f^{r^n}$ is the $n$th-order residual of $f$, if it exists, and $f^{\ell^n}$ is the $n$th-order dual residual of $f$, if it exists.
    For any chain $\mathbf{\Omega}$, \emph{$F(\mathbf{\Omega})$} denotes the set of all (order-preserving) functions on $\m \Omega$ with residuals of all orders and dual residuals of all orders.

\subsection{Integral chains and limit points}

Given a chain $\mathbf{\Omega}$, we say that a point $a\in\Omega$  \emph{has a $k$-cover} (or \emph{has a cover of order $k$}), if $k \in \mathbb{Z}^+$ and there exist  $a_{1},a_{2}\ldots,a_{k}\in\Omega $ with  $a \prec a_1 \prec \cdots \prec a_{k}$, or if $k \in \mathbb{Z}^-$ and there exist $a_{-1},a_{-2}\dots,a_{k}\in\Omega $ with  $a_k \prec \cdots \prec a_{-1} \prec  a$. Note that these elements are unique, since $\mathbf{\Omega}$ is a chain.  In this case, we call $a_k$ the \emph{$k$-cover} of $a$ and denote it by $a + k$; we also define $a + 0 = a$.

An element in a chain $\m \Omega$ is said to be \emph{integral} if it has a $k$-cover for all $k \in \mathbb{Z}$; equivalently, it belongs to an interval of $\m \Omega$ that is isomorphic to the integers. A chain $\mathbf{\Omega}$ is said to be \emph{integral} if all of its elements are integral; equivalently, $\m \Omega$ is a lexicographic product of $\mathbb{Z}$ with some chain.

\medskip

Surprisingly, when the chain is integral,  if a function is residuated and dually residuated, then it has residuals and dual residuals of all orders. This demonstrates one of the benefits of working with integral chains.

\begin{lemma}\label{l:1&2 in integral}
 If $f$ is a map on an integral chain $\mathbf{\Omega}$, then  $f\in F(\mathbf{\Omega})$ iff $f$ is residuated and dually residuated.
\end{lemma}
\begin{proof}
For the non-trivial direction, we prove first that if a map $f:\Omega\rightarrow\Omega$ is dually residuated (hence $f^\ell$ exists), then $f^{\ell}$ is also dually residuated (hence $f^{\ell \ell}$ exists). By  Lemma~\ref{l: bounded preimage}, it suffices to show that $f^{\ell}$ satisfies (1) and (2) of the lemma.

(1) Suppose $a\in f^{\ell}[\Omega]$. If $|(f^{\ell})^{-1}[ a ]|=1$,  then $(f^{\ell})^{-1}[ a]=\{b\}=[b,b]$ for some $b\in\Omega$.
If $|(f^{\ell})^{-1}[a]|>1$, there exist $x,y\in\Omega$ such that $x<y$ and $f^{\ell}(x)=f^{\ell}(y)=a$. 
Then  $x\notin f[\Omega]$, because otherwise $x=f(b)$ for some $b$ and, by Lemma~\ref{l: facts of fl and fr}(1,2), $x<y$ iff $f(b)<y$ iff $ff^{\ell}f(b)<y$ iff $f^{\ell}f(b)<f^{\ell}(y)$ iff $f^{\ell}(x)<f^{\ell}(y)$, which is a contradiction. 
So, by Lemma~\ref{l: facts of fl and fr}(4), $f^r(x)\prec f^{\ell}(x)$. We will show that $(f^{\ell})^{-1}[a]=[ff^r(x)+1,f(a)]$. Indeed, using Lemma~\ref{l: facts of fl and fr}(2) we get that for all $z\in\Omega$,
$z \in (f^{\ell})^{-1}[a]$
iff $f^\ell(z)=a$ 
iff  ($a\leq f^{\ell}(z)$ and $f^{\ell}(z)\leq a$) 
iff ($f^{\ell}(x)\leq f^{\ell}(z)$ and $z\leq f(a)$) 
iff ($f^r(x)<f^{\ell}(z)$ and $z\leq f(a)$) 
iff ($ff^r(x)<z\leq f(a)$) 
iff $ff^r(x)+1\leq z\leq f(a)$
iff $z \in [ff^r(x)+1,f(a)]$. 

(2) Suppose $a\not \in f^{\ell}[\Omega]$. 
Clearly, $f(a)\prec f(a)+1$; we will show that $a\in(f^{\ell}f(a),f^{\ell}(f(a)+1))$. Since $a\notin f^{\ell}[\Omega]$ and $f^{\ell}f(a)\leq a$, we get that $f^{\ell}f(a)<a$. Also,  $f(a) < f(a)+1$ yields $a < f^{\ell}(f(a)+1)$ by Lemma~\ref{l: facts of fl and fr}(2).

A dual argument shows that for a map $f$, if $f^r$ exists then $f^{rr}$ also exists. Now, a simple induction shows that if $f$ is residuated and dually residuated, then residuals and dual residuals of $f$ of all orders exist, hence $f\in F(\mathbf{\Omega})$. 
\end{proof}

 \begin{corollary}\label{c: intF(Omega)}
  If $f$ is a map on an integral chain $\mathbf{\Omega}$, then  $f\in F(\mathbf{\Omega})$ iff $f$ is order-preserving and satisfies the two conditions of Lemma~\ref{l: bounded preimage}.
 \end{corollary}

In integral chains, every element has a lower cover and an upper cover, so no element is a limit point. In arbitrary chains, the existence of limit points creates complications which we will address. 
 Given a chain $\mathbf{\Omega}$, we define:
\[
\Omega^{-}	:=\{ a\in\Omega:\,a=\underset{b<a}{\bigvee}b\},
  \quad 
  \Omega^{+}	:=\{ a\in\Omega:\,a=\underset{a<b}{\bigwedge}b\} 
 \;\; \text{ and } \;\; 
\Omega^{\pm}:=\Omega^{-}\cup\Omega^{+},
\]
the sets of \emph{limit points from below}, \emph{limit points from above} and \emph{limit points}, respectively.

The following lemma shows that maps that are residuated and dually residuated on a chain $\m \Omega$ preserve and reflect
certain elements of $\Omega^-$ and $\Omega^+$.

\begin{lemma}\label{l:about limit points}
 Let $f$ be  a residuated and dually residuated map on a chain $\mathbf{\Omega}$ and $a\in\Omega$.
\begin{enumerate}
  \item  If $a=f^{\ell}f(a)$, then: $a\in\Omega^{-}$ iff $f(a)\in\Omega^{-}$. 

   \item  If $a=f^{r}f(a)$, then: $a\in\Omega^{+}$ iff $f(a)\in\Omega^{+}$. 
 
 \item  We have $f^{\ell}(a)\in\Omega^{-}$  iff  ($a \in f[\Omega]$ and $a\in\Omega^{-}$).
 \item We have $f^r(a)\in\Omega^{+}$  iff  ($a \in f[\Omega]$ and $a\in\Omega^{+}$).

\end{enumerate}
\end{lemma}
\begin{proof} Recall that since $f$ is residuated, it preserves existing joins. Also, since $f$ is dually residuated, $f^\ell$ exists and $f^\ell$ is residuated (with residual $f$), so $f^\ell$ preserves existing joins, as well. We will prove (1) and (3) as the proofs of (2) and (4) are dual.

(1) For the forward direction, using $a\in \Omega^{-}$, that $f$ preserves existing joins, $a=f^{\ell}f(a)$ and Lemma~\ref{l: facts of fl and fr}(2), we get that
$f(a)=f(\bigvee\{ b\in \Omega :b<a\})
=\bigvee\{ f(b) :b<a\} =\bigvee\{ f(b) :b< f^{\ell}f(a)\} 
=\bigvee\{ f(b) :f(b)<f(a)\}\leq$
$\bigvee\{ x\in \Omega :x<f(a)\}\leq f(a)$. So $f(a) \in \Omega ^{-}$.

Conversely, using  $f(a)\in\Omega^-$, $f^{\ell}$ preserves existing joins and Lemma~\ref{l: facts of fl and fr}(2), we get
$a=f^{\ell}f(a)=f^{\ell} [\bigvee\{ x\in \Omega :x<f(a)\}]
=\bigvee\{ f^{\ell}(x)\in \Omega :x<f(a)\} 
=\bigvee\{ f^{\ell}(x)\in \Omega :f^{\ell}(x)<a\} 
\leq \bigvee\{ b\in \Omega :b<a\}\leq a$, so $a=\bigvee\{ b\in \Omega :b<a\}$ and $a \in \Omega ^{-}$.

(3) If $f^{\ell}(a)\in\Omega^{-}$, the element $f^{\ell}(a)$ has no lower cover and in particular $f^r(a)\not \prec f^{\ell}(a)$. By Lemma~\ref{l: facts of fl and fr}(4), we get that $a \in f[\Omega]$, so $a=ff^{\ell}(a)$ by Lemma~\ref{l: facts of fl and fr}(1).
Also, by Lemma~\ref{l: facts of fl and fr}(1), $f^{\ell}f(f^{\ell}(a))=f^{\ell}(a)$, so by instantiating (1) for $f^{\ell}(a)$ we get $a=f(f^{\ell}(a))\in \Omega^-$. 

For the converse, since $a\in f[\Omega]$, Lemma~\ref{l: facts of fl and fr}(1) yields $ff^{\ell}(a)=a\in\Omega^-$. Also, by Lemma~\ref{l: facts of fl and fr}(1) we have $f^{\ell}(a)=f^{\ell}ff^{\ell}(a)$, so by instantiating (1) for $f^{\ell}(a)$ we get $f^{\ell}(a)\in \Omega^-$.
\end{proof}

The following lemma shows that certain types of limit points cannot occur at locations where the function is many-to-one (where multiple elements have the same image). This distinction of whether the function is many-to-one or one-to-one at a given element will end up being a recurring theme in this section.

\begin{lemma}\label{l:F(O) bounded images} 
If $\m \Omega$ is a chain, $a\in \Omega$, $f\in F( \mathbf{\Omega} )$ and $|f^{-1}[ a]| >1$, then $f^{\ell}(a)\notin\Omega^{-}$,  $f^r(a)\notin\Omega^{+}$ and $a\notin \Omega^{\pm}$.

\end{lemma}
\begin{proof}
We first prove the contraposition: $f^{\ell}(a)\in\Omega^{-}$ implies $|f^{-1}[ a]|\leq 1$. By Lemma~\ref{l:about limit points}(3), $f^{\ell}(a)\in\Omega^{-}$ yields $a \in f[\Omega]$ and $a\in\Omega^{-}$; hence $ff^{\ell}(a)=a\in\Omega^{-}$, by Lemma \ref{l: facts of fl and fr}(1). By Lemma~\ref{l: facts of fl and fr}(3), we have $|f^{-1}[ a ]|=[f^{\ell} (a),f^{r}(a)]$ and $f^{\ell} (a) \leq f^{r}(a)$, so to get $f^{-1}[ a]\leq 1$ it suffices to prove that $f^{r}(a)\leq f^\ell(a)$. 
Using  $a\in\Omega^-$, the fact that $f^r$ preserves existing joins, $a=ff^{\ell}(a)$ and Lemma \ref{l: facts of fl and fr}(2), we get 
 $f^r(a)=f^r(\bigvee\{x:x<a\})=\bigvee\{f^r(x):x<ff^{\ell}(a) \}
 = \bigvee\{f^r(x):f^r(x)<f^{\ell}(a)\}\leq f^{\ell}(a)$. 

 Likewise, we show that:  $f^{r}(a)\in\Omega^{+}$ implies $|f^{-1}[ a]|\leq 1$. Also, if we have $|f^{-1}[ a]| >1$, we showed above that  $f^{\ell}(a)\notin\Omega^{-}$ and  $f^r(a)\notin\Omega^{+}$. By Lemma~\ref{l:about limit points}(3,4), we get $a\notin \Omega^{\pm}$.
\end{proof}

\subsection{The embedding theorem}
  It is shown in \cite{GJKO} that, given a chain $\m \Omega$, 
 $F(\mathbf{\Omega})$ gives rise to a distributive $\ell$-pregroup $\mathbf{F(\Omega)}$, under functional composition and pointwise order. 
Also,  every distributive $\ell$-pregroup embeds into $\mathbf{F(\Omega)}$, for some  chain $\m \Omega$ by \cite{GH}.
We will improve this representation by showing that for every chain $\m \Omega$, there exists an integral chain $\overline{\m \Omega}$, such that  $\mathbf{F(\Omega)}$  embeds in $\mathbf{F(\overline{\Omega})}$. We first give an example on how the chain $\overline{\m \Omega}$ is constructed and how the embedding works. 

\begin{example}\label{e: extension}
 Suppose that for a chain $\mathbf{\Omega}$ we have $\{a,\gamma,f(a)\}\subseteq \Omega^-$ and $\gamma\in\Omega^+$, as shown in Figure~\ref{f:omega bar}. 
The chain $\overline{\Omega}$ will contain
\[
\Omega\cup \{ (-n,a):n\in\mathbb{Z}^{+}\}\cup \{ (-n,f(a)):n\in\mathbb{Z}^{+}\} \cup \{ (n,\gamma):n\in\mathbb{Z}^{*}\}
\]
We define $(k,a)\leq (n,b)$ if and only if  $a<b$ or ($a=b$ and $k\leq n$), where we write $(0,a)$ for $a \in \Omega$. 
 Also, given a function $f:\Omega\rightarrow\Omega$, as  in Figure~\ref{f:omega bar}, $\overline{f}$  
 extends $f$ with $\overline{f}((n,a))=(n,f(a))$ and $\overline{f}((n,a))=f(\gamma)$ in an order-preserving way. 
 \end{example}
 
 \begin{figure}[ht]\def\eq{=} 
 \begin{center}
{\scriptsize
\begin{tikzpicture}
[scale=0.5]
\node[fill,draw,circle,scale=0.3,left](2) at (0,2){};
\node[left](2.1) at (-0.2,2){$c$};
\node[fill,draw,circle,scale=0.3,right](2r) at (2,2.5){};
\node[right](2.1) at (2.1,2.5){$f(c)$};
\node[fill,draw,circle,scale=0.3,left](1) at (0,1){};
\node[left](1.1) at (-0.2,1){$b$};
\node[fill,draw,circle,scale=0.3,right](1r) at (2,1.5){};
\node[right](2.1) at (2.1,1.5){$f(b)$};
\node[fill,draw,red,circle,scale=0.3,left](0) at (0,0){};
\node[left](0.1) at (-0.2,0){$a$};
\node[fill,draw,red,circle,scale=0.3,right](0r) at (2,0.5){};
\node[right](0.1) at (2.1,0.5){$f(a)$};

\node[fill,draw,circle,scale=0.3,left](-5) at (0,-4){};
\node[left](-4.1) at (-0.2,-4){$\alpha$};

\node[fill,draw,circle,scale=0.3,left](-6) at (0,-8){};
\node[left](-8.1) at (-0.2,-8){$\beta$};
\node[fill,draw,circle,scale=0.3,right](-6r) at (2,-7.5){};
\node[right](-8.1) at (2.1,-7.5){$f(\gamma)$};
\node[fill,draw,red,circle,scale=0.3,left](g) at (0,-6){};
\node[left](-6.1) at (-0.2,-6){$\gamma$};
\node at (0)[below=-1pt]{$\vdots$};
\node at (0r)[below=-1pt]{$\vdots$};
\node at (-6)[below=-1pt]{$\vdots$};
\node at (-6r)[below=-1pt]{$\vdots$};
\node at (-5)[above=3pt]{$\vdots$};
\node at (2)[above=3pt]{$\vdots$};
\node at (2r)[above=3pt]{$\vdots$};
\node at (g)[above=1pt]{$\vdots$};
\node at (g)[below=-4pt]{$\vdots$};
\draw[-](2)--(2r);
\draw[-](1)--(1r);
\draw[-](0)--(0r);
\draw[-](-5)--(-6r);
\draw[-](-6)--(-6r);
\draw[-](g)--(-6r);
\node at (1,-9){\normalsize$f$};
\end{tikzpicture}
\qquad
\begin{tikzpicture}
[scale=0.5]
\node[fill,draw,circle,scale=0.3,left](2) at (0,2){};
\node[left](2.1) at (-0.2,2){$c$};
\node[fill,draw,circle,scale=0.3,right](2r) at (2,2.5){};
\node[right](2.1) at (2.1,2.5){$f(c)$};
\node[fill,draw,circle,scale=0.3,left](1) at (0,1){};
\node[left](1.1) at (-0.2,1){$b$};
\node[fill,draw,circle,scale=0.3,right](1r) at (2,1.5){};
\node[right](2.1) at (2.1,1.5){$f(b)$};
\node[fill,draw,circle,scale=0.3,left](0) at (0,0){};
\node[left](0.1) at (-0.2,0){$a$};
\node[fill,draw,circle,scale=0.3,right](0r) at (2,0.5){};
\node[right](0.1) at (2.1,0.5){$f(a)$};
\node[left](-1) at (-0.2,-1){$(-1,a)$};
\node[fill,draw,circle,scale=0.3,right](-1r) at (2,-0.5){};
\node[right](-1.1) at (2.1,-0.5){$(-1,f(a))$};
\node[fill,draw,circle,scale=0.3,left](-1) at (0,-1){};
\node[fill,draw,circle,scale=0.3,left](-2) at (0,-2){};
\node[left](-2.2) at (-0.2,-2){$(-2,a)$};
\node[fill,draw,circle,scale=0.3,right](-2r) at (2,-1.5){};
\node[right](-2.1) at (2.1,-1.5){$(-2,f(a))$};

\node[fill,draw,circle,scale=0.3,left](-5) at (0,-4){};
\node[left](-4.1) at (-0.2,-4){$\alpha$};
\node[fill,draw,circle,scale=0.3,left](g+) at (0,-5.5){};
\node[left](-5.1) at (-0.2,-5.3){$(1,\gamma)$};

\node[fill,draw,circle,scale=0.3,left](g) at (0,-6){};
\node[left](-6.1) at (-0.5,-6){$\gamma$};

\node[fill,draw,circle,scale=0.3,left](g-) at (0,-6.5){};
\node[left](-7.1) at (-0.2,-6.5){$(-1,\gamma)$};

\node[fill,draw,circle,scale=0.3,left](-6) at (0,-8){};
\node[left](-8.1) at (-0.2,-8){$\beta$};
\node[fill,draw,circle,scale=0.3,right](-6r) at (2,-7.5){};
\node[right](-8.1) at (2.1,-7.5){$f(\gamma)$};
\node at (-2)[below=-1pt]{$\vdots$};
\node at (-2r)[below=-1pt]{$\vdots$};
\node at (-6)[below=-1pt]{$\vdots$};
\node at (-6r)[below=-1pt]{$\vdots$};
\node at (2)[above=3pt]{$\vdots$};
\node at (2r)[above=3pt]{$\vdots$};
\node at (g+)[above=1pt]{$\vdots$};
\node at (g-)[below=-4pt]{$\vdots$};
\draw[-](2)--(2r);
\draw[-](1)--(1r);
\draw[-](0)--(0r);
\draw[-](-2)--(-2r);
\draw[-](-1)--(-1r);
\draw[-](-5)--(-6r);
\draw[-](-6)--(-6r);
\draw[-](g)--(-6r);
\draw[-](g+)--(-6r);
\draw[-](g-)--(-6r);
\node at (1,-9){\normalsize$\overline{f}$};
\end{tikzpicture}
}
\caption{Extension to an integral chain.
}
\label{f:omega bar}
\end{center}
\end{figure}
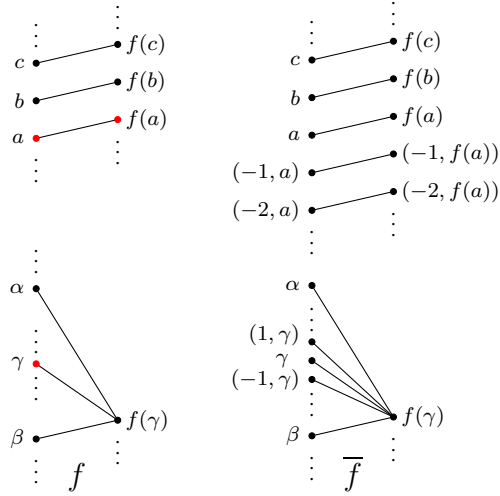
 In general, given a chain $\mathbf{\Omega}$, we define the extended chain $\mathbf{\overline{\Omega}}$ as follows:
\[
\overline{\Omega}=\Omega\cup\{ (k,a):a\in\Omega^{+}\text{ and }k\in\mathbb{Z}^{+}\} \cup\{ (-k,a):a\in\Omega^{-}\text{ and }k\in\mathbb{Z}^{+}\} \]
To simplify notation we set $ka:=(k,a)$, when $(k,a)\in\overline{\Omega}$ and $k \not = 0$; 
also, to keep the notation uniform we set $0a:=a$ for $a\in\Omega$. So, the all elements of $\overline{\Omega}$ are of the form $ka$.
The order in $\mathbf{\overline{\Omega}}$ is defined by:  
\begin{center}
   $ka\leq nb$ iff  $a < b$ or ($a=b$ and $k\leq n$), for $ka,nb\in \overline{\Omega}$. 
\end{center}

\begin{lemma}
 If $\m \Omega$ is a chain, then $\mathbf{\overline{\Omega}}$ is an integral chain that extends $\m \Omega$.
\end{lemma}

\begin{proof}
 Every point $a\in \Omega - \Omega^+$ has a cover in $\m \Omega$ and the same element is a cover in $\mathbf{\overline{\Omega}}$; also, every point $ka\in  \Omega^+ \cup (\mathbb{Z}^+ \times \Omega^+)\cup (\mathbb{Z}^- \times \Omega^-)$ is covered by $(k+1)a $ in $\mathbf{\overline{\Omega}}$. Likewise, every element in $\mathbf{\overline{\Omega}}$ has a lower cover.     
\end{proof}

Note that in Example~\ref{e: extension}, the value of $\overline{f}$ on a new element, such as $(-1) a$ or $(-1)\gamma$, depends on whether the function is injective at $a$ or at $\gamma$; here, we say that $f$ is injective at $b$, if $|f^{-1}[f(b)]|=1$. This leads to the following definition.

\medskip

     If $\m \Omega$ is a chain, $f\in F(\mathbf{\Omega})$ and $kb\in \overline{\Omega}$, we define 
     \begin{itemize}
         \item $\overline{f}(kb)=f(b)$, if $|f^{-1}[f (b)]|>1$ or $k=0$.
         \item $\overline{f}(kb)=kf(b)$, if $|f^{-1}[f(b)]|=1$ and  $k \not = 0$.
     \end{itemize}

\begin{lemma}\label{l:barfell}
     If $\m \Omega$ is a chain and $f\in F(\mathbf{\Omega})$, then
     $\overline{f}\in F(\mathbf{\overline{\Omega}})$.
\end{lemma}

\begin{proof}
 First observe that in the second case of the definition, $kf(b)$ indeed exists in $\overline{\Omega}$: Since $k \not = 0$, we have $b \in \Omega^\pm$, and since $|f^{-1}[f(b)]|=1$, we have $f^{\ell}f(b)=f^{r}f(b)=b$, by Lemma~\ref{l: facts of fl and fr}(3). So, by Lemma~\ref{l:about limit points}(1,2), we know that $f(b)\in\Omega^-$, if $b\in\Omega^-$, and $f(b)\in\Omega^+$, if $b\in\Omega^+$; so $kf(b)$ exists in $\overline{\Omega}$. Therefore, $\bar{f}$ is a function on $\mathbf{\overline{\Omega}}$.
Also, note that $\overline{f}$ is  order-preserving. Indeed, assume that   $ka,nb\in \overline{\Omega}$ and $ka < nb$, i.e.,  $a < b$  or ($a=b$ and $k < n$). If $a=b$ and $k < n$, then $f(a)=f(b)$, so:  
$\overline{f}(ka)= f(a) = f(b)=\overline{f}(nb)$ in the  case $|f^{-1}[f (a)]|>1$; 
$\overline{f}(ka)= f(a)=f(b) \leq \overline{f}(nb)$ in  case $0=k<n$; 
$\overline{f}(ka) \leq f(a)=f(b) = \overline{f}(nb)$ in  case $k<n=0$; 
and $\overline{f}(ka) = kf(a)=kf(b)< nf(b)=\overline{f}(nb)$ in  case $|f^{-1}[f (a)]|=1$ and $k \not = 0 \not = n$. 
If $a < b$, then $f(a) \leq f(b)$.  In the case $f(a) < f(b)$, we have $kf(a) < nf(b)$ and so $\overline{f}(ka)<\overline{f}(nb)$. In the case $f(a) = f(b)$, we have $|f^{-1}[f (a)]|>1$, so $\overline{f}(ka)= f(a) = f(b)=\overline{f}(nb)$.  

In view of Corollary~\ref{c: intF(Omega)}, to establish $\overline{f}\in F(\mathbf{\overline{\Omega}})$ it suffices to show that $\overline{f}$ satisfies (1) and (2) of Lemma~\ref{l: bounded preimage}.

(1) Suppose $a\in\overline{f}[\overline{\m \Omega}]$; we will show that $(\overline{f})^{-1}[a]=[c,d]_{\mathbf{\overline{\Omega}}}$, for $c,d \in \overline{\Omega}$. 
 If $a\notin\Omega$, then $a=kf(b)$ for some $k \not = 0$ and $b\in\Omega^{\pm}$, where $|f^{-1}[ f(b)]|=1$, i.e., $b$ is unique; it  follows by the definition of $\overline{f}$ that $(\overline{f})^{-1}[a]=(\overline{f})^{-1}[kf(b)]=\{kb\}$. 

 If $a\in\Omega$, then $a\in f[\Omega]$; we will show $(\overline{f})^{-1}[a]=  [f^{\ell} (a),f^{r}(a)]_{\mathbf{\overline{\Omega}}}$.
 If $kb \in  (\overline{f})^{-1}[a]$, then by the definition of  $\overline{f}$ we see that  $b \in \Omega$,  $f(b)=a$ and $|f^{-1}[f (b)]|>1$ or $k=0$. Also, since $a\in f[\Omega]$,  Lemma~\ref{l: facts of fl and fr}(3) yields $f^{-1}[ a ]=[f^{\ell} (a),f^{r}(a)]_{\mathbf{\Omega}}$ and $f^{\ell} (a) \leq f^{r}(a)$; hence $b \in f^{-1}[ a ]= [f^{\ell} (a),f^{r}(a)]_{\mathbf{\Omega}}$. 
 If $k=0$, then $kb=b \in [f^{\ell} (a),f^{r}(a)]_{\mathbf{\Omega}} \subseteq [f^{\ell} (a),f^{r}(a)]_{\mathbf{\overline{\Omega}}}$. If $k \not = 0$ (i.e., $b \in \Omega^\pm$), then $|f^{-1}[f (b)]|>1$, and by Lemma~\ref{l:F(O) bounded images}, $f^{\ell}(a)\notin\Omega^-$ and $f^r(a)\notin\Omega^+$. So, if $b \in \Omega^-$, then $f^{\ell}(a)< b \leq f^r(a)$ and $k<0$, hence $f^{\ell}(a) < kb \leq f^r(a)$; if $b \in \Omega^+$, then $f^{\ell}(a)\leq b < f^r(a)$ and $k>0$, hence $f^{\ell}(a) \leq kb < f^r(a)$. Therefore,  $(\overline{f})^{-1}[a] \subseteq [f^{\ell} (a),f^{r}(a)]_{\mathbf{\overline{\Omega}}}$. Conversely, if $kb \in [f^{\ell} (a),f^{r}(a)]_{\mathbf{\overline{\Omega}}}$, then $f^{\ell}(a)\leq kb \leq f^r(a)$. Since $a\in f[\Omega]$ and $\overline{f}$ is order-preserving, we have $a=f(f^\ell(a))=\overline{f}(f^\ell(a)) \leq \overline{f}(kb) \leq \overline{f}(f^r(a))=f(f^r(a))=a$; hence $\overline{f}(kb)=a$ and $kb \in \overline{f}^{-1}[a]$. So, $(\overline{f})^{-1}[a]=  [f^{\ell} (a),f^{r}(a)]_{\mathbf{\overline{\Omega}}}$.

(2) If $a\notin \overline{f}[\overline{\Omega}]$, we will show  $f^{r}(a)\prec_{\mathbf{\overline{\Omega}}} f^{\ell}(a)$ and $a\in(\overline{f} (f^{r}(a)), \overline{f}(f^{\ell}(a)))_{\mathbf{\overline{\Omega}}}$.
Since $f^r(a) \in \Omega$ and $f^{\ell}(a) \in \Omega$, we have  $\overline{f}f^r(a)=  ff^r(a)$   and $\overline{f} f^{\ell}(a)=ff^{\ell}(a)$.

If $a\in\Omega$, then since $a\notin\overline{f}[\overline{\Omega}]$, it follows that $a\notin f[\Omega]$.
 So, by Lemma~\ref{l: facts of fl and fr}(4), $f^r(a)\prec_{\mathbf{\Omega}}f^{\ell}(a)$ and $a\in (ff^r(a), ff^{\ell}(a))_{\mathbf{\Omega}}$; thus
 $a\in (ff^r(a), ff^{\ell}(a))_{\mathbf{\overline{\Omega}}}$. 
 Since $f^r(a)\prec_{\mathbf{\Omega}}f^{\ell}(a)$, we have $f^r(a) \not \in \Omega^+$ and $f^{\ell}(a) \not \in \Omega^-$; hence $f^r(a)\prec_{\m{\overline{\Omega}}}f^{\ell}(a)$.
 
If $a\notin \Omega$, then there exist $k \not = 0$ and $b\in\Omega^{\pm}$ such that $a=kb$; hence $|f^{-1}[ b]| \leq 1$, by Lemma~\ref{l:F(O) bounded images}. If we had $b \in f[\Omega]$, then we would get $|f^{-1}[ b]| = 1$, say $b=f(c)$,
and $f^\ell f(c)=c=f^r f(c)$ by Lemma~\ref{l: facts of fl and fr}(3). So by Lemma~\ref{l:about limit points}(1,2), we would get ($c\in\Omega^-$, if $b\in\Omega^-$) and ($c\in\Omega^+$, if $b\in\Omega^+$). Then, by the definition of $\overline{f}$, we would get $\overline{f}(kc)=kf(c)=kb=a$, contradicting  $a\notin \overline{f}[\overline{\Omega}]$; thus,
  $b\notin f[\Omega]$. 
  As above, we get 
$f^r(b)\prec_{\m{\overline{\Omega}}}f^{\ell}(b)$
 and $b\in (ff^r(b), ff^{\ell}(b))_{\mathbf{\overline{\Omega}}}$. 
 Since $ff^r(b)<b< ff^{\ell}(b)$, we get $ff^r(b)<kb< ff^{\ell}(b)$, hence $a=kb\in (ff^r(b), ff^{\ell}(b))_{\mathbf{\overline{\Omega}}}$. 
\end{proof}

\begin{theorem} \label{t:F(omega) F(overline(omega))}
    For every chain $\m{\Omega}$, the assignment $\overline{\text{}\cdot\text{}}:\mathbf{F(\Omega)}\rightarrow\mathbf{F(\overline{\Omega})}$ is an $\ell$-pregroup embedding.
\end{theorem}

\begin{proof}
 If $\overline{f}=\overline{g}$, for some $f,g\in F(\m \Omega)$, then their restrictions to $\Omega$ are equal, hence $f=g$. Thus $f \mapsto \overline{f}$ is injective.
Also, by Lemma \ref{l:barfell}, we know that for all $f\in F(\m \Omega)$, $\overline{f^{\ell}}=(\overline{f})^{\ell}$ and $\overline{f^{r}}=(\overline{f})^{r}$.

To prove that for all $f,g\in F(\m \Omega)$, $\overline{f\circ g}=\overline{f}\circ\overline{g}$,  we will verify that $\overline{f\circ g}(kb)=\overline{f}\circ\overline{g}(kb)$, for all $kb\in\overline{\Omega}$. If $k=0$, then $kb=b \in \Omega$, so $g(b)\in\Omega$ and $\overline{(f\circ g)}(b)=(f\circ g)(b)=fg(b)=\overline{f}g(b)=\overline{f}\overline{g}(b)=(\overline{f}\circ \overline{g})(b)$. If $k \not = 0$, then $b \in \Omega^\pm$. 
If $|(f\circ g)^{-1}[(f\circ g)(b)]|=1$, then $|f^{-1}[fg(b)]|=1$ and
$|g^{-1}[g(b)]|=1$, so 
$(\overline{f}\circ\overline{g})(kb)=
\overline{f}\overline{g}(kb)=
\overline{f}(kg(b))=
kf(g(b))=
k(f\circ g)(b)=
\overline{f\circ g}(kb)$. 
If $|(f\circ g)^{-1}[(f\circ g)(b)]|>1$  and $|f^{-1}[fg(b)]|=1$ then $|g^{-1}[g(b)]|>1$ and  
$(\overline{f}\circ\overline{g})(kb)=
\overline{f}\overline{g}(kb)=
\overline{f}(g(b))=fg(b)=\overline{f\circ g}(kb)$.
Finally, if we have $|(f\circ g)^{-1}[(f\circ g)(b)]|>1$, 
and $|f^{-1}[fg(b)]|>1$, then $(\overline{f}\circ\overline{g})(kb)=
\overline{f}\overline{g}(kb)=
\overline{f}(n g(b))=
fg(b)=
(f\circ g)(b)=
\overline{(f\circ g)}(kb)$, where $n \in \{k, 0\}$. 

To check that $\overline{\text{}\cdot\text{}}:\mathbf{F(\Omega)}\rightarrow\mathbf{F(\overline{\Omega})}$ preserves meets, let $kb \in \overline{\Omega}$; the proof that it preserves joins is dual. If $k=0$, then $\overline{f\wedge g}(b)= (f\wedge g)(b)=f(b)\wedge g(b)= \overline{f}(b)\wedge \overline{g}(b)=(\overline{f} \wedge \overline{g})(b)$.
 If $k \not = 0$, then $b \in \Omega^\pm$. 
 We will consider several cases.

If  $|f^{-1}[f(b)]|=1$ and $|g^{-1}[g(b)]|=1$, then $|(f\wedge g)^{-1}[(f\wedge g)(b)]|=1$: If $c,c'\in\Omega$ with $c<b<c'$, then $f(c)<f(b)<f(c')$ and $g(c)<g(b)<g(c')$, hence $(f\wedge g)(c)<(f\wedge g)(b)<(f\wedge g)(c')$. Therefore, $\overline{f\wedge g}(kb)=k(f\wedge g)(b)=kf(b)\wedge kg(b)=\overline{f}(a)\wedge\overline{g}(kb)=(\overline{f}\wedge\overline{g})(kb)$.   

If  $|f^{-1}[f(b)]|>1$ and $|g^{-1}[g(b)]|>1$, then $|(f\wedge g)^{-1}[(f\wedge g)(b)]|>1$: By Lemma \ref{l: facts of fl and fr}(3), we have $b \in f^{-1}[f(b)]=[f^{\ell}(f(b)),f^r(f(b))]$ and  $b \in g^{-1}[g(b)]=[g^{\ell}(g(b)),g^r(g(b))]$ and by Lemma~\ref{l:F(O) bounded images} we have $f^{\ell}(f(b)), g^{\ell}(g(b))\not \in \Omega^-$ and also $f^r(f(b)), g^r(g(b)) \not \in \Omega^+$. 
So, in the case $b \in \Omega^-$, we have 
$f^{\ell}(f(b))$, $g^{\ell}(g(b))< b$, so
$\max\{f^{\ell}(f(b)), g^{\ell}(g(b))\}<b\leq \min\{f^r(f(b)),g^r(g(b))\}$ and since $b \in \Omega^-$, 
$\max\{f^{\ell}(f(b)), g^{\ell}(g(b))\}<c< b\leq \min\{f^r(f(b)),g^r(g(b))\}$, for some $c$.
So,  $b,c \in f^{-1}[f(b)]\cap g^{-1}[g(b)]$, hence 
$f(c)=f(b)$ and $g(c)=g(b)$; therefore $(f \mt g)(c)=f(c) \mt g(c)=f(b)\mt g(b)=(f \mt g)(b)$. Consequently, $|(f\wedge g)^{-1}[(f\wedge g)(b)]|>1$, so  $\overline{f\wedge g}(kb)=(f\wedge g)(b)=f(b)\wedge g(b)=\overline{f}(kb)\wedge\overline{g}(kb)=(\overline{f}\wedge\overline{g})(kb)$. The case where $b \in \Omega^+$ is similar. 
 
Now we consider the case where one function is injective at $b$ and the other is not. Without loss of generality, 
we assume that $|f^{-1}[f(b)]|>1$ and $|g^{-1}[g(b)]|=1$. Then  Lemma \ref{l:F(O) bounded images} yields $f(b) \notin \Omega^{\pm}$ and, since $b \in \Omega^\pm$ and $g^\ell(g(b))=b$, Lemma \ref{l:about limit points}(1,2) yields $g(b) \in \Omega^{\pm}$; hence $f(b) \not = g(b)$. 

If $g(b) < f(b)$, then $kg(b)<f(b)$. We show that $|(f\wedge g)^{-1}[(f\wedge g)(b)]|=1$, from which it follows that $\overline{f\wedge g}(kb)=k(f\wedge g)(b)=kg(b)=f(b)\wedge kg(b)=\overline{f}(kb)\wedge\overline{g}(kb)=(\overline{f}\wedge\overline{g})(kb)$. 
It suffices to show that, for all $c,c'\in \Omega$ such that $ c<b<c'$, we have $(f\wedge g)(c)<(f\wedge g)(b)<(f\wedge g)(c')$. 
From $|g^{-1}[g(b)]|=1$ we get $g(c)<g(b)<g(c')$, hence $g(c)<f(b)$. Therefore, we have
 $f(c) \mt g(c) \leq f(b) \mt g(c)=g(c)<
 g(b)=f(b)\wedge g(b)=g(b)<f(b) \mt g(c') \leq f(c') \mt g(c')$.

 If $f(b) < g(b)$, then  $f(b)<kg(b)$. We show that $|(f\wedge g)^{-1}[(f\wedge g)(b)]| > 1$, from which it follows that $\overline{f\wedge g}(kb) = (f\wedge g)(b) = f(b) = f(b) \wedge kg(b) = \overline{f}(kb) \wedge \overline{g}(kb) = (\overline{f} \wedge \overline{g})(kb)$. We distinguish two cases about what type of limit $b$ is.

If $b \in \Omega^+$, then $b <  f^rf(b)$, because $f(b) \notin \Omega^{\pm}$ implies $f^rf(b) \notin \Omega^{+}$ by Lemma~\ref{l:about limit points}(4). 
By Lemma~\ref{l: facts of fl and fr}(1), we have
$f(b) = ff^rf(b)= f(c)$, where $c:=f^rf(b)$.
Also, since $g$ is order-preserving and $|g^{-1}[g(b)]|=1$, we get $g(b) < g(c)$. 
Consequently, we have $b \not = c$ and
$(f\wedge g)(b) =f(b) \mt g(b)=f(b) \mt g(b) \mt g(c)= f(b) \mt g(c) = f(c) \mt g(c)= (f\wedge g)(c)$. 

If $b \in \Omega^-$, by the order preservation of $g$, we get 
$f(b)<g(b)=g(\bigvee \{d: d<b\})= \bigvee \{g(d): d<b\})$. So there is a $d<b$ such that $f(b)<g(d)$. Moreover, by Lemma~\ref{l:about limit points}(3), $f(b) \notin \Omega^{-}$ implies $f^\ell(f(b)) \notin \Omega^{\pm}$ and since $b \in \Omega^-$, we get $f^\ell(f(b))<b$. Taking $c:=d \jn f^\ell(f(b))$, we get 
 $f^\ell(f(b))\leq c<b$ and $f(b)<g(c)$. By Lemma~\ref{l: facts of fl and fr}(1) and the order-preservation of $f$, we get $f(b)=ff^\ell(f(b))\leq f(c) \leq f(b)$, so $f(b) = f(c)$. Also, since $g$ is order-preserving and $|g^{-1}[g(b)]|=1$, we get $g(c) < g(b)$. 
Consequently, we have $b \not = c$ and
$(f\wedge g)(b) =f(b) \mt g(b)=f(b) \mt g(b) \mt g(c)= f(b) \mt g(c) = f(c) \mt g(c)= (f\wedge g)(c)$. 

To show that the map also preserves inverses, let $f \in F(\m \Omega)$. Then, since $\m F(\m \Omega)$ is an $\ell$-pregroup, we have 
$
f^{\ell} f\leq1\leq ff^{\ell} 
$, so by applying the map, we get 
$
(\overline{f^{\ell}}) \overline{f} \leq 1\leq \overline{f} (\overline{f^{\ell}}) 
$. By applying residuation in $\m F(\overline{\m \Omega})$ we get 
$
\overline{f^{\ell}}  \leq 
(\overline{f}){^\ell}
\leq 
\overline{f^{\ell}}
$. Therefore, $(\overline{f})^{\ell}=\overline{f^{\ell}}$, and likewise $(\overline{f})^{r}=\overline{f^{r}}$.
\end{proof}

Together with the embedding theorem of \cite{GH} for distributive $\ell$-pregroups, Theorem~\ref{t:F(omega) F(overline(omega))} yields the following improved embedding theorem.

\begin{corollary}\label{c: embedding omega integral}
    Every distributive $\ell$-pregroup can be embedded in $\m F(\mathbf{\Omega})$ for some integral chain $\mathbf{\Omega}$.
\end{corollary}

\begin{corollary}\label{c: failure in F(O) integral}
If an equation fails in $\mathsf{DLPG}$, then it fails in  $\mathbf{F(\Omega)}$, for some integral chain $\m \Omega$.
\end{corollary}

\section{From $\mathbf{F(\Omega)}$ to a diagram}\label{s: diagrams}

In this section, we show how the failure of an equation in $\m F(\m \Omega)$, where $\m \Omega$ is an integral chain, can be encoded into a finite object that we call a diagram. We will provide a running example in order to illustrate the main ideas, but first, we show that equations over distributive $\ell$-pregroups can be written in a suitable form.

\medskip 

For convenience, given an $\ell$-pregroup $\mathbf{A}=(A,\wedge,\vee,\cdot,1,^{\ell} ,^{r})$ and $x\in A$, we set $x^{(0)}:=x$ and
for all $m\in\mathbb{Z}^{+}$, $x^{(m)}:= x^{\ell^m}$, $x^{(-m)}:=x^{r^{m}}$. 

Note that every  equation $s=t$ in the language of
$\ell$-pregroups is equivalent to the conjunction of the equations $s \leq t$ and $t \leq s$, which by residuation  is equivalent to the conjunction of  $1\leq s^r t$ and $1 \leq t^r s$, and finally to the equation
$1\leq s^r t\wedge t^r s$; i.e., to an equation of the form $1 \leq q$.
Also, since in $\mathsf{DLPG}$ the maps $^\ell$ and $^r$ are anti-homomorphisms for multiplication
$$(x\cdot y)^{\ell} 	=y^{\ell} \cdot x^{\ell} \qquad
(x\cdot y)^{r}	=y^{r}\cdot x^{r}$$
and the De Morgan laws 
$$(x\wedge y)^{\ell}= x^{\ell} \vee y^{\ell} \quad
(x\vee y)^{\ell}= x^{\ell} \wedge y^{\ell} \quad
(x\wedge y)^{r}=x^{r}\vee y^{r}\quad
(x\vee y)^{r}=x^{r}\wedge y^{r}
$$
hold, all applications of the inverses $^\ell$ and $^r$ in $1 \leq q$ can be pushed down to the variables. Now, using lattice distributivity and the fact that multiplication distributes over meet, the equation $1 \leq q$ is equivalent to  $1 \leq q_1 \mt \cdots \mt q_k$, where the $q_i$'s do not use meet and all inverses are still pushed to the variables. Hence, the equation is equivalent to the conjunction   $1 \leq q_1$ and $\ldots$ and  $1 \leq q_k$, where we may assume that the variables in $q_i$ are different from the variables in $q_j$, by renaming variables. Therefore, this conjunction is equivalent to the equation $1 \leq q_1 \cdots q_k$ (the backward direction is obtained by setting all variables not in $q_i$ equal to $1$). Therefore, the equation is equivalent to one of the form $1 \leq w$, where the term $w$ does not involve meets and all inverses are pushed to the variables. Since multiplication distributes over join, this equation can in turn be written as 
$$1\leq w_{1}\vee\ldots\vee w_{n},$$
 where the $w_i$'s are \emph{intentional terms}: they involve only multiplication and inverses and the inverses have been pushed to the variables. In other words, intentional terms are of the form $$x_{1}^{(m_1)}x_{2}^{(m_2)}\ldots x_{k}^{(m_k)}$$
 where $x_1,\ldots, x_k$ are variables (not necessarily distinct), $k\in \mathbb{N}$  and $m_1, \ldots, m_k \in \mathbb{Z}$. We say that the equation displayed above is in \emph{(disjunctive) intentional form}. We have, therefore,  established the following result.

\begin{lemma}
   Every equation over  $\mathsf{DLPG}$ is equivalent to one in intentional form.
\end{lemma}

\begin{example}\label{e: diagram}
As a motivating example, we consider a simple equation in intentional form. The equation $1 \leq x^{\ell}x$ fails in  $\mathbf{F(\mathbb{Z})}$, because for the function $f \in \mathbf{F(\mathbb{Z})}$ given in the Figure \ref{f:building D}
 \begin{figure}[ht]\def\eq{=}
 \begin{center}
{\scriptsize
\begin{tikzpicture}
[scale=0.5]
\node[fill,draw,circle,scale=0.3,left](2) at (0,2){};
\node[left](2.1) at (-0.2,2){$8$};
\node[fill,draw,circle,scale=0.3,right](2r) at (2,2){};
\node[right](2.1) at (2.1,2){$8$};
\node[fill,draw,circle,scale=0.3,left](1) at (0,1){};
\node[left](1.1) at (-0.2,1){$7$};
\node[fill,draw,circle,scale=0.3,right](1r) at (2,1){};
\node[right](2.1) at (2.1,1){$7$};
\node[fill,draw,circle,scale=0.3,left](0) at (0,0){};
\node[left](0.1) at (-0.2,0){$6$};
\node[fill,draw,circle,scale=0.3,right](0r) at (2,0){};
\node[right](0.1) at (2.1,0){$6$};
\node[fill,draw,circle,scale=0.3,left](-1) at (0,-1){};
\node[left](-1.1) at (-0.2,-1){$5$};
\node[fill,draw,circle,scale=0.3,right](-1r) at (2,-1){};
\node[right](-1.1) at (2.1,-1){$5$};
\node[fill,draw,circle,scale=0.3,left](-2) at (0,-2){};
\node[left](-2.1) at (-0.2,-2){$4$};
\node[fill,draw,circle,scale=0.3,right](-2r) at (2,-2){};
\node[right](-2.1) at (2.1,-2){$4$};
\node[fill,draw,circle,scale=0.3,left](-3) at (0,-3){};
\node[left](-3.1) at (-0.2,-3){$3$};
\node[fill,draw,circle,scale=0.3,right](-3r) at (2,-3){};
\node[right](-3.1) at (2.1,-3){$3$};
\node[fill,draw,circle,scale=0.3,left](-4) at (0,-4){};
\node[left](-4.1) at (-0.2,-4){$2$};
\node[fill,draw,circle,scale=0.3,right](-4r) at (2,-4){};
\node[right](-4.1) at (2.1,-4){$2$};

\node at (-4)[below=-1pt]{$\vdots$};
\node at (-4r)[below=-1pt]{$\vdots$};
\node at (2)[above=3pt]{$\vdots$};
\node at (2r)[above=3pt]{$\vdots$};
\node at (1,-5){\normalsize$f$};
\draw[-](2)--(1r);
\draw[blue,->](1)--(-1r);
\draw[-](0)--(-1r);
\draw[-](-1)--(-1r);
\draw[-](-1)--(-1r);
\draw[red,->](-2)--(-1r);
\draw[red,->](-3)--(-4r);
\end{tikzpicture}
\qquad
\begin{tikzpicture}
[scale=0.5]

\node[fill,draw,circle,scale=0.3,left](2) at (0,2){};
\node[left](2.1) at (-0.2,2){$8$};
\node[fill,draw,circle,scale=0.3,right](2r) at (2,2){};
\node[right](2.1) at (2.1,2){$8$};
\node[fill,draw,circle,scale=0.3,left](1) at (0,1){};
\node[left](1.1) at (-0.2,1){$7$};
\node[fill,draw,circle,scale=0.3,right](1r) at (2,1){};
\node[right](2.1) at (2.1,1){$7$};
\node[fill,draw,circle,scale=0.3,left](0) at (0,0){};
\node[left](0.1) at (-0.2,0){$6$};
\node[fill,draw,circle,scale=0.3,right](0r) at (2,0){};
\node[right](0.1) at (2.1,0){$6$};
\node[fill,draw,circle,scale=0.3,left](-1) at (0,-1){};
\node[left](-1.1) at (-0.2,-1){$5$};
\node[fill,draw,circle,scale=0.3,right](-1r) at (2,-1){};
\node[right](-1.1) at (2.1,-1){$5$};
\node[fill,draw,circle,scale=0.3,left](-2) at (0,-2){};
\node[left](-2.1) at (-0.2,-2){$4$};
\node[fill,draw,circle,scale=0.3,right](-2r) at (2,-2){};
\node[right](-2.1) at (2.1,-2){$4$};
\node[fill,draw,circle,scale=0.3,left](-3) at (0,-3){};
\node[left](-3.1) at (-0.2,-3){$3$};
\node[fill,draw,circle,scale=0.3,right](-3r) at (2,-3){};
\node[right](-3.1) at (2.1,-3){$3$};
\node at (-3)[below=-1pt]{$\vdots$};
\node at (-3r)[below=-1pt]{$\vdots$};
\node at (2)[above=3pt]{$\vdots$};
\node at (2r)[above=3pt]{$\vdots$};
\node at (1,-5){\normalsize$f^{\ell}$};
\draw[-](1)--(2r);
\draw[-](0)--(2r);
\draw[red,->](-1)--(-2r);
\draw[-](-2)--(-2r);
\draw[-](-3)--(-2r);
\end{tikzpicture}
\qquad
\begin{tikzpicture}
[scale=0.5]

\node[fill,draw,circle,scale=0.3,left](1) at (0,1){};
\node[left](1.1) at (-0.2,1){$7$};
\node[fill,draw,circle,scale=0.3,right](1r) at (2,1){};
\node[right](2.1) at (2.1,1){$7$};
\node[fill,draw,circle,scale=0.3,left](-1) at (0,-1){};
\node[left](-1.1) at (-0.2,-1){$5$};
\node[fill,draw,circle,scale=0.3,right](-1r) at (2,-1){};
\node[right](-1.1) at (2.1,-1){$5$};
\node[fill,draw,circle,scale=0.3,left](-2) at (0,-2){};
\node[left](-2.1) at (-0.2,-2){$4$};
\node[fill,draw,circle,scale=0.3,right](-2r) at (2,-2){};
\node[right](-2.1) at (2.1,-2){$4$};
\node[fill,draw,circle,scale=0.3,left](-2) at (0,-2){};
\node[left](-3.1) at (-0.2,-3){$3$};
\node[fill,draw,circle,scale=0.3,right](-3r) at (2,-3){};
\node[right](-3.1) at (2.1,-3){$3$};
\node[fill,draw,circle,scale=0.3,left](-3) at (0,-3){};
\node[left](-4.1) at (-0.2,-4){$2$};
\node[fill,draw,circle,scale=0.3,right](-4r) at (2,-4){};
\node[right](-4.1) at (2.1,-4){$2$};
\node[fill,draw,circle,scale=0.3,left](-4) at (0,-4){};

\node at (1,-5){\normalsize$\mathbf{D}$};
\draw[blue,->](1)--(-1r);
\draw[red,->](-2)--(-1r);
\draw[red,->](-3)--(-4r);
\end{tikzpicture}
}
\caption{ Building a diagram from $f$}
\label{f:building D}
 \end{center}
\end{figure}
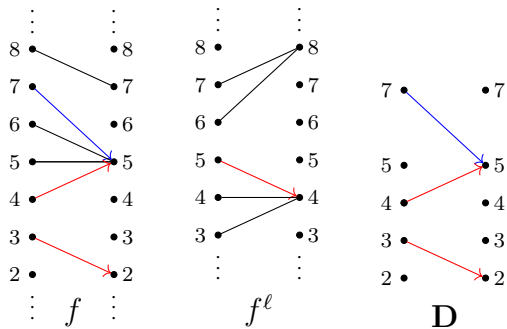
we have  $id_\mathbb{Z} \not \leq f^{\ell}f$ in $\mathbf{F(\mathbb{Z})}$, since $f^{\ell}f(7)=4<7= id_\mathbb{Z}(7)$. However, we can witness this  failure in a much smaller chain than $\mathbb{Z}$ and by keeping only part of $f$ and $f^\ell$: we can take the chain consisting of the three elements $7, 5, 4$ 
and also we can restrict $f$ and $f^\ell$ to partial functions $g$ and $g^{[\ell]}$ consisting only of $g(7)=5$ and $g^{[\ell]}(5)=4$; here we use the notation $g^{[\ell]}$ as the calculation will be done on the finite chain (and $g^\ell$ on a finite chain exists only when $g$ is the identity). As any failure should be witnessed only by a chain and a (partial) function $g$, without any reference to $g^{[\ell]}$, we could further translate the information $g^{[\ell]}(5)=4$ into information about $g$. For that we need to find a way to define $g^{[\ell]}$ for a partial function $g$ on a finite chain in a formal way. Also, this has to be done in a way that the computation of the partial function $g^{[\ell]}$ on the finite chain agrees with the computation of  $f^\ell$ on $\mathbb{Z}$, at least on the points that are of interest to us: $g^{[\ell]}(5)=4$. It turns out that to make sure that these computations agree, we need to include more elements in our finite chain and also define $g$ on some of these elements.  Lemma~\ref{l: bounded preimage} showcases the importance of covering pairs in the existence of $g^{[\ell]}$, so we will need to mark some pairs of elements of the finite chain as covering pairs, thus motivating the definition of a c-chain below.
\end{example}

\subsection{Diagrams} 
Failures of equations in $\m F(\m \Omega)$, where $\m \Omega$ is a chain, give rise to structures that we call diagrams, based on a chain of elements. 

 Given a chain $\mathbf{\Omega}$, a partial function $g$ over $\Omega$ is called \emph{order-preserving} if, for all $a,b\in Dom(g)$, $a\leq b$ implies $g(a)\leq g(b)$.

 A \emph{c-chain} is a triple $(\Delta,\leq,\diagcov)$, consisting of a finite chain $(\Delta,\leq)$ and a subset ${\diagcov}\subseteq{\prec}$ of the covering relation, i.e. if $a\diagcov b$, then $a$ is covered by $b$.
Every subset $\Delta'$ of a c-chain $\mathbf{\Delta}=(\Delta,\leq,\diagcov)$ defines a \emph{sub c-chain} $\mathbf{\Delta'}=(\Delta',\leq,\diagcov)$ of $\mathbf{\Delta}$, under the restrictions of the relations.

Clearly, every chain $\m \Omega$ gives rise to the c-chain  $(\mathbf{\Omega}, \prec)$, where $\prec$ is the covering relation of  $\m \Omega$, but other subsets of $\prec$ result in different c-chains on $\m \Omega$.

\medskip

Inspired by Lemma~\ref{l: bounded preimage}, given  a c-chain $\mathbf{\Delta}=(\Delta,\leq,\diagcov)$ and
a partial function $g$ on $\Delta$, we define the relation $g^{[\ell]}$ on $\Delta$ by: for all $x, b\in\Delta$, $(x,b) \in g^{[\ell]}$ iff $b\in Dom(g)$ and there exists $a\in Dom(g)$ such that $a\diagcov b$ and $g(a)<x\leq g(b)$. 
Also, we define the relation $g^{[r]}$ on $\Delta$ by: for all $x, a\in\Delta$, $(x,a) \in g^{[r]}$ iff $a\in Dom(g)$, there exists $b\in Dom(g)$, such that $a\diagcov b$ and $g(a)\leq x< g(b)$.

\begin{lemma}
If $\m \Delta$ is a c-chain and $g$ is an order-preserving partial function on $\m \Delta$, then $g^{[\ell]}$ and $g^{[r]}$ are order-preserving partial functions over $\m \Delta$. 
\end{lemma}

\begin{proof}
  If $(x,b), (y,b') \in g^{[\ell]}$ and $x \leq y$, then there exist $a,a' \in Dom(g)$ such that $a\diagcov b$, $g(a)<x\leq g(b)$, $ a' \diagcov b'$ and $g(a')<y\leq g(b')$; therefore,  $g(a)< g(b')$.
Note that $b' \leq a$ leads to the contradiction $g(b') \leq g(a)$, by the order preservation of $g$, so $a < b'$, since $\m \Omega$ is a chain. Now, since  $a \prec b$ and $a<b'$, we have $b \leq b'$.   

In particular,   $(x,b), (x,b') \in g^{[\ell]}$ implies $b'=b$; consequently $g^{[\ell]}$ is a partial function. As  usual, we write $g^{[\ell]}(x)=b$ for $(x,b) \in g^{[\ell]}$.
Rewriting the first paragraph in this notation yields that if $x \leq y$, then $g^{[\ell]}(x) \leq g^{[\ell]}(y)$. Thus, $g^{[\ell]}$ is order preserving.
\end{proof}

We mention that it can be shown (following the proofs of Lemmas~\ref{l:1&2 in integral} and \ref{l: bounded preimage})  that if $\m \Omega$ is an integral chain and $f$ an order-preserving function on $\m \Omega$, then the lower residual $f^\ell$ of $f$ exists on $\m \Omega$ iff the partial function $f^{[\ell]}$ on the c-chain $(\m{\Omega}, \prec)$ is a total function, and also in this case the two coincide. This is why we have chosen to use a similar notation for $f^\ell$ and $f^{[\ell]}$. However, we will only apply this notion for finite chains.

 Given a c-chain $\m \Delta$, and an order-preserving partial function $g$ on $\m \Delta$, we define $g^{[n]}$, for all $n \in \mathbb{N}$, recursively by:  $g^{[0]}=g$ and $g^{[k+1]}:=(g^{[k]})^{[\ell]}$. Also,  we define $g^{[-n]}$, for all $n \in \mathbb{N}$, recursively by $g^{[-(k+1)]}:=(g^{[k]})^{[r]}$.

\medskip

Having defined a way to compute (iterated) inverses of a partial function $g$ on a c-chain, our next goal is to make sure that we make the c-chain and the partial function $g$ large enough so that the computation of $g^{[\ell]}$ is performed correctly on the points of interest. Going back to our Example~\ref{e: diagram}, we see that to make sure that $g^{[\ell]}(5)=4$, we need to include the values $g(4)=f(4)=5$ and $g(3)=f(3)=2$, together with the covering relation $3 \diagcov 4$. Using the notation $-a$ for the lower cover of $a$, we see that $g$ needs to be defined on 
    $4=f^{\ell}f(7)$ and on $3=-f^{\ell}f(7)$,
 and it yields the values 
 $5=ff^{\ell}f(7)$ and $2=f-f^{\ell}f(7)$.
 
So, the sets of points (expressed in terms of the original point $7$) guaranteeing that $g(7)$ and $g^{[\ell]}(5)$ are computed correctly are respectively:
$$\Delta_{f,0}^7:=\{ 7, f(7) \}=\{7,5\}$$
$$ \Delta_{f,1}^{f(7)}: =\{f(7), f^{\ell}f(7), -f^{\ell}f(7), f-f^{\ell}f(7)\}=\{5, 4, 3, 2\}$$

\medskip 

    To simplify the notation in the following definition and subsequent results, we write 
    $-a$ and also $(-1)a$ for $a-1$, the lower cover of $a$; likewise we write $+a$ and $1a$ for $a+1$, the upper cover of $a$. Finally, we write $0a$ for $a$.

 Given an integral chain $\mathbf{\Omega}$, $f\in F(\mathbf{\Omega})$, $a\in\Omega$ and $m\in\mathbb{\mathbb{N}}$, we define the sets:
\begin{align*}
\Lambda^a_{f,m} & :=\{ \sigma_{1}f^{(1)}\ldots\sigma_{m}f^{(m)}(a):\,\sigma_1,\ldots,\sigma_{m}\in\{-1,0\}\}\\
\Delta_{f,m}^a & :=\{ a\} \cup \bigcup_{j=0}^{m} \{ \sigma_{j}f^{(j)}\ldots\sigma_{m}f^{(m)}(a):\,\sigma_j,\ldots,\sigma_{m}\in\{-1,0\},\sigma_0=0\}\\
\Lambda^a_{f,-m} & :=\{ \sigma_{1}f^{(-1)}\ldots\sigma_{m}f^{(-m)}(a):\,\sigma_{1},\ldots,\sigma_{m}\in\{1,0\}\}\\
\Delta^a_{f,-m} & :=\{ a\} \cup \bigcup_{j=0}^{m} \{ \sigma_{j}f^{(-j)}\ldots\sigma_{m}f^{(-m)}(a):\,\sigma_{j},\ldots,\sigma_{m}\in\{1,0\},\sigma_0=0\}
\end{align*}
In particular, when $m=0$, $\Lambda^a_{f,0}$ must be interpreted as $\{a\}$ and $\Delta_{f,0}^a$ as $\{a,f(a)\}$. Also, given an integral chain $\mathbf{\Omega}$, $f\in F(\mathbf{\Omega})$, $a\in\Omega$ and $m\in\mathbb{\mathbb{Z}}$ we define the sub c-chain $\mathbf{\Delta}^a_{f,m}=(\Delta^a_{f,m},\leq,\prec)$ of $(\mathbf{\Omega},\prec)$.

\medskip

The next two lemmas, show that the iterated inverses of a partial function on a finite c-chain are computed correctly, provided we take the chain to include the points considered above.

\begin{lemma}\label{l: Delta subf,m,a}
If $\mathbf{\Omega}$ is an integral chain,  $f\in F(\mathbf{\Omega})$, $a\in\Omega$, $m\in\mathbb{\mathbb{Z}}$,
$\mathbf{\Delta}$ is a sub c-chain of $(\mathbf{\Omega},\prec)$ containing $\Delta^a_{f,m}$, and $g$ is an order-preserving partial function over $\mathbf{\Delta}$ such that $g|_{\Lambda^a_{f,m}}=f|_{\Lambda^a_{f,m}}$, then $g^{[m]}(a)=f^{(m)}(a)$.
\end{lemma}
\begin{proof}
    We will show the result for $m\in\mathbb{N}$, since the proof for the case $m\in\mathbb{Z}^-$ is dual.
    Using induction on $i$, we will prove that for all $0\leq i\leq m$:  $g^{[i]}(x)=f^{(i)}(x)$, for all $x=\sigma_{i+1}f^{(i+1)}\ldots\sigma_{m}f^{(m)}(a)$, where $\sigma_{i+1},\ldots,\sigma_{m}\in\{-1,0\}$. Note that  $g^{[m]}(a)=f^{(m)}(a)$ is the special case of the statement for $i=m$, as in that case we have $\sigma_{i+1}f^{(i+1)}\ldots\sigma_{m}f^{(m)}(a)=a$.
    
    If $\sigma_1,\ldots,\sigma_{m}\in\{-1,0\}$ and $x=\sigma_{1}f^{(1)}\ldots\sigma_{m}f^{(m)}(a)$, then $x\in\Lambda^a_{f,m}$, so  $g^{[0]}(x)=g(x)=f(x)=f^{(0)}(x)$.
    
    Assume the statement holds for $i=n<m$: for all  $\sigma_{n+1},\ldots,\sigma_{m}\in\{-1,0\}$ and $x=\sigma_{n+1}f^{(n+1)}\ldots\sigma_{m}f^{(m)}(a)$, we have $g^{[n]}(x)=f^{(n)}(x)$. Now let $\sigma_{n+2}$, $\ldots$, $\sigma_{m}\in\{-1,0\}$, and  $y=\sigma_{n+2} f^{(n+2)}\ldots\sigma_{m}f^{(m)}(a)$; we will show that $g^{[n+1]}(y)=f^{(n+1)}(y)$. Since $-1, \sigma_{n+2},\ldots,\sigma_{m}\in\{-1,0\}$ and $0, \sigma_{n+2},\ldots,\sigma_{m}\in\{-1,0\}$, two applications of the inductive hypothesis give $g^{[n]}(-f^{(n+1)}(y))=f^{(n)}(-f^{(n+1)}(y))$ and also $g^{[n]}(f^{(n+1)}(y))=f^{(n)}(f^{(n+1)}(y))$. 
    So for $G:=g^{[n]}$ and $F:=f^{(n)}$, we have $G(-F^{(\ell)}(y))=F(-F^{(\ell)}(y))$ and $G(F^{(\ell)}(y))=F(F^{(\ell)}(y))$. We will show that $G^{[\ell]}(y)=F^{\ell}(y)$, which implies  $g^{[n+1]}(y)=G^{[\ell]}(y)=F^{\ell}(y)=f^{(n+1)}(y)$, as desired.
    Indeed, we have $F^{\ell}(y)=\min \{ b\in \Omega :y\leq F (b) \}$, so  $y\leq F(F^{\ell}(y))$ and $y\nleq F(-F^{\ell}(y))$; hence $F(-F^{\ell}(y))<y$, given that $\mathbf{\Omega}$ is a chain. So $G(-F^{\ell}(y))=F(-F^{\ell}(y))<y\leq F(F^{\ell}(y))=G(F^{\ell}(y))$, thus $G^{[\ell]}(y)=F^{\ell}(y)$, by the definition of $G^{[\ell]}$.
\end{proof}

 A \emph{diagram} $(\m \Delta, f_{1},\ldots,f_{n})$, consists of a finite c-chain $\m \Delta$ and  order-preserving partial functions $f_{1},\ldots,f_{n}$ on $\m \Delta$, where $n\in\mathbb{N}$.

 \subsection{Failure of an equation and syntax} We will define what we mean by the failure of an equation on a diagram. 
 First recall that a failure of an equation $s=t$ in $\m F(\m \Omega)$ consists of a homomorphism $\varphi: \m {Tm} \rightarrow \m F(\m \Omega)$, from the algebra $\m {Tm}$ of all terms in the language of $\ell$-pregroups, such that $\varphi(s) \not = \varphi(t)$; clearly only the values of $\varphi$ on  terms in the variables involved in the equation matter. Since equations can be written in intentional form  $1\leq w_{1}\vee\ldots\vee w_{k}$, and $\varphi(1) \not \leq \varphi( w_{1}\vee\ldots\vee w_{k})$ is equivalent to $\varphi(1) \not \leq \varphi( w_{1}) \vee\ldots\vee \varphi(w_{k})$, we need only assume that $\varphi$ is an intentional homomorphism (in the language of multiplication and inverses). Finally, since $\m F(\m \Omega)$  consists of functions,  $\varphi(1) \not \leq \varphi( w_{1}) \vee\ldots\vee \varphi(w_{k})$ if and only if there exist $p \in \Omega$ such that $\varphi(1)(p) \not \leq \varphi( w_{1})(p) \vee\ldots\vee \varphi(w_{k})(p)$, and since $\m \Omega$ is a chain this is equivalent to $\varphi(1)(p) > \varphi( w_{1})(p), \ldots  , \varphi(w_{k})(p)$.
 We will ask something similar for failure in a diagram, but we need an intentional algebra playing the role of $\m F(\m \Omega)$. 

 Given a c-chain $\m \Delta$, we define the algebra $\m {Pf}(\m \Delta)=(Pf(\m \Delta), {\circ}, ^{[\ell]}, ^{[r]}, i_\Delta)$, where $Pf(\m \Delta)$ is the set of all the order-preserving partial functions over $\m \Delta$, $\circ$ is  the composition of partial functions/relations, $i_\Delta$ is the identity function on $\Delta$ and  $f \mapsto f^{[\ell]}$ and $f \mapsto f^{[r]}$ are the two inversion operations as defined on  $\m \Delta$. 

We say that the equation $1\leq w_{1}\vee\ldots\vee w_{k}$ in intentional form over variables $x_1, \ldots, x_n$  \emph{fails} in a diagram  $(\m \Delta, f_{1},\ldots,f_{n})$ if there is an intentional homomorphism $\varphi: \m {Ti} \rightarrow \m {Pf}(\m \Delta)$, from the algebra $\m {Ti}$ of all intentional terms, and a point $p \in \Delta$ such that   $\varphi(1)(p) > \varphi( w_{1})(p), \ldots  , \varphi(w_{k})(p)$ and $\varphi(x_i)=f_i$ for all $1 \leq i\leq n$.

\medskip

 Given a  c-chain $\m \Delta$, note that the relation $\diagcov$ is actually itself an order-preserving partial function on $\m \Delta$. Instead of writing ${\diagcov}(a)=b$ for $a\diagcov b$, we write $+(a)=b$; so when viewing $\diagcov$ as an element of  $Pf(\m \Delta)$ we denote this partial function by $+$. Likewise, the converse of the relation $\diagcov$ is also an order-preserving partial function, and we denote it by $-$, when viewing it as an element of $Pf(\m \Delta)$:  $-(a)=b$ iff $b\diagcov a$. We denote by $\m {Pf}(\m \Delta)^{\pm}$ the expansion of the algebra $\m {Pf}(\m \Delta)$ with these two distinguished elements $+$ and $-$. Also, every intentional homomorphism  $\varphi: \m {Ti} \rightarrow \m {Pf}(\m \Delta)$ extends to a homomorphism $\varphi^\pm: \m {Ti}^{\pm} \rightarrow \m {Pf}(\m \Delta)^{\pm}$, where $\m {Ti}^{\pm}$ is the expansion with the two constants. Likewise, when $\m \Omega$ is an integral chain, we can consider homomorphisms $\varphi^\pm: \m {Ti}^{\pm} \rightarrow \m F(\m \Omega)^{\pm}$, where $\m F(\m \Omega)^{\pm}$ denotes the expansion of $\m F(\m \Omega)$ with the functions $+$ and $-$ defined by $+x=x+1$ and $-x=x-1$. 
 This expansion will be useful in the proof of Theorem~\ref{t: DLPG to comsur}, when we will be making an inductive argument in the expanded language (with $+$ and $-$).

\medskip

Given an equation $\varepsilon$ in  intentional form $1\leq w_{1}\vee\ldots\vee w_{k}$ we will define a purely syntactic object consisting of a large enough set $\Delta_\varepsilon$ of terms that will serve as a template for evaluations into a sufficient collection of failing diagrams for $\varepsilon$. First we
define the set of \emph{final subwords} of $\varepsilon$
  \[
 FS_\varepsilon:=\{ u: w_1=vu\text{ or  } ... \text{ or  } w_k=vu, \text{ for some } u\} 
  \]
  We actually take $\m {Ti}$ to satisfy the equalities $v1=v=1v$, so strictly speaking we take it to be a quotient of the absolutely free algebra. With this understanding, 
 $FS_\varepsilon$ contains $1$.
  Also, for a variable $x$ among the variables $x_1, \ldots, x_n$ of  $\varepsilon$, $m\in\mathbb{N}$ and $v\in FS_\varepsilon$ we define the following sets; they are syntactic analogues of the $\Delta$'s that were defined relative to a chain $\m \Omega$ and an $f \in F(\m \Omega)$.
  \begin{align*}
 \Delta_{x,m}^{v} & :=\{ v\} \cup \bigcup_{j=0}^{m} \{ \sigma_{j}x^{(j)}\ldots\sigma_{m}x^{(m)}v:\,\sigma_j,\ldots,\sigma_{m}\in\{-1,0\},\sigma_0=0\}\\
\Delta_{x,-m}^{v} & :=\{ v\} \cup \bigcup_{j=0}^{m} \{ \sigma_{j}x^{(-j)}\ldots\sigma_{m}x^{(-m)}v:\,\sigma_{j},\ldots,\sigma_{m}\in\{1,0\},\sigma_0=0\}\\
S_\varepsilon&:=\{ (i,m,v):i\in\{1,\ldots,n\},m\in\mathbb{Z},v\in FS_\varepsilon\text{ and }x_{i}^{(m)}v\in FS_\varepsilon\} \\
\Delta_{\varepsilon}&:=\{1\} \cup \underset{(i,m,v)\in S}{\bigcup}\Delta_{x_{i},m}^{v}
\end{align*}

Observe that if $w \in FS_\varepsilon$, then $w=1$ (in which case it is in $\Delta_\varepsilon$) or it is of the form $x_i^{(m)}v$, where $v \in FS_\varepsilon$, $i \in \{1, \ldots, n\}$ and $m \in \mathbb{Z}$. So, $x_i^{(m)}v \in \Delta_{x_i,m}^{v}$ (it appears for the case where $j=m$ in the union in the formula for $\Delta_{x_i,m}^{v}$ for $\sigma_m=0$).
Therefore, $FS_\varepsilon \subseteq \Delta_\varepsilon$. When the equation $\varepsilon$ is understood from the context we will omit the subscripts in $FS_\varepsilon$ and $S_\varepsilon$. Clearly $FS_\varepsilon$ is finite. Also, since there are only finitely many $m 
\in \mathbb{Z}$ such that $x_i^{(m)}$ appears in $\varepsilon$, the set $S_\varepsilon$ is finite, hence also $\Delta_\varepsilon$ is finite.

\subsection{Compatible surjections}

The following concept turns out to be a useful intermediate step between a failure of an equation in $DLPG$ and a failure in a diagram. Also, it will provide a convenient language to express the algorithm that yields decidability for $\mathsf{DLPG}$.

Given an equation $\varepsilon(x_1,\ldots,x_n)$ in intentional form, 
a \emph{compatible surjection} for $\varepsilon$ is  an onto map $\varphi:\Delta_{\varepsilon}\rightarrow\mathbb{N}_q$, where $\mathbb{N}_q=\{1, \ldots, q\}$ is
 an initial segment of $\mathbb{Z}^+$ under the natural order (and $q\leq |\Delta_\varepsilon|$), such that:
\begin{itemize}
    \item[(i)] The relation $g_i:=\{(\varphi(u),\varphi(x_iu)) \mid u, x_iu\in\Delta_{\varepsilon}\}$ on $\mathbb{N}_q$ is an order-preserving partial function for all $i \in \{1,\ldots,n\}$.
    \item[(ii)]  The relation ${\diagcov} :=\{(\varphi(v), \varphi(+v)) \mid v, +v\in\Delta_{\varepsilon}\} \cup \{(\varphi(-v),\varphi(v)) \mid v, -v\in\Delta_{\varepsilon}\}$ on $\mathbb{N}_q$ is contained in the covering relation $\prec$ of $\mathbb{N}_q$.
    \item[(iii)] $\varphi(x_i^{(m)}u)=g_i^{[m]}(\varphi(u))$, when $i \in \{1,\ldots,n\}$, $m\in\mathbb{Z}$ and $u, x_i^{(m)}u\in\Delta_{\varepsilon}$.
\end{itemize}

Note that the first two conditions ensure that $\m D_{\varepsilon, \varphi}:= (\mathbb{N}_q, {\leq}, {\diagcov}, g_1, \ldots, g_n)$ is a diagram; in (iii),  $g_i^{[m]}$ is calculated in this diagram. Also, it follows that the relation $\diagcov$ is an order-preserving partial function. The most interesting feature of compatible surjections is that even though they seem to simply assign to $u \in \Delta_\varepsilon$ a value $\varphi(u)$ in a chain $\mathbb{N}_q$, they also indirectly assign to $u \in \Delta_\varepsilon$ a partial function on the diagram $(\mathbb{N}_q, \leq, \diagcov)$: the last condition states that the  function $\varphi_\varepsilon:\Delta_{\varepsilon}\rightarrow \m {Pf}(\mathbb{N}_q, {\leq}, {\diagcov})^\pm$ is a intentional homomorphism (between partial algebras), where $\varphi_\varepsilon(u)(\varphi(v))=\varphi(uv)$, for $uv \in \Delta_\varepsilon$. Therefore, terms in $\Delta_\varepsilon$ yield not only the elements of a chain but also the partial functions on this chain.

We say that the equation $1\leq w_{1}\vee\ldots\vee w_{k}$ in intentional form \emph{fails} in a compatible surjection $\varphi$, if $\varphi(w_1), \ldots, \varphi(w_k)<\varphi(1)$.

\begin{theorem}\label{t: DLPG to comsur}
    If an equation $\varepsilon$ fails in $\mathsf{DLPG}$, then it also fails in some compatible surjection for $\varepsilon$. 
\end{theorem}
\begin{proof}
Let $\varepsilon=\varepsilon(x_1,\ldots,x_n)$ be an equation $1\leq w_{1}\vee\ldots\vee w_{k}$ in intentional form  that fails in  $\mathsf{DLPG}$.
By Corollary~\ref{c: failure in F(O) integral},  $\varepsilon$ fails in $\mathbf{F}(\mathbf{\Omega})$, for some integral chain $\mathbf{\Omega}$. Therefore, there exists a list $f=(f_1,\ldots,f_n)$ of elements of ${F}(\mathbf{\Omega})$ and $p\in\Omega$  such that $w_{1}^{\m{F(\Omega)}}(f)(p), \ldots,  w_{k}^{\m{F(\Omega)}}(f)(p)>p$; here $f_i=\psi(x_i)$ where $\psi:\m {Ti} \ra \m F(\m \Omega)$ is the homomorphism witnessing the failure of $\varepsilon$. We denote by $\psi^\pm:\m {Ti}^\pm \ra \m F(\m \Omega)^\pm$ the extension of $\psi$.
We will show that $\psi_p:\Delta_{\varepsilon}\rightarrow\psi_p[\Delta_{\varepsilon}]$ is a compatible surjection, where the order on $\psi_p[\Delta_{\varepsilon}]$ is inherited from $\m \Omega$ and 
$$\psi_p(u):=\psi^\pm(u)(p)=u^{\m{F(\Omega)^\pm}}(f)(p)$$ 
for $u \in \Delta_\varepsilon$; actually, the compatible surjection will be the composition of $\psi_p$ with the (unique) isomorphism of the chain $\psi_p[\Delta_{\varepsilon}]$ with the initial segment $\mathbb{N}_q$ of $\mathbb{Z}^+$, where $q=|\psi_p[\Delta_{\varepsilon}]|$. To simplify the notation, we write $u(f)(p)$ and  $u_{fp}$ for $u^{\m{F(\Omega)^\pm}}(f)(p)$.

Clearly, $\psi_p$ is onto. To show that $g_i$ is an order-preserving partial function, suppose $i\in \{1,\ldots,n\}$, $u,v,x_i u,x_i v\in\Delta_{\varepsilon}$ and $\psi_p(u)\leq \psi_p(v)$. Since $f_i$ is order preserving, $g_i(\psi_p(u))=\psi_p(x_i u)=(x_i u)(f)(p)=f_i(u(f)(p))=f_i(\psi_p(u))\leq f_i (\psi_p(v))=f_i(v(f)(p))=(x_i v)(f)(p)=\psi_p(x_i v)=g_i(\psi_p(v))$. In particular, $g_i$ is the restriction of $f_i$ to the domain of $g_i$.

If $a \diagcov b$, then $a=\psi_p(v)$ and $b=\psi_p(+v)$ for some $v, +v\in\Delta_{\varepsilon}$,  or $a=\psi_p(-v)$ and $b=\psi_p(v)$ for some $v, -v\in\Delta_{\varepsilon}$. In the first case 
$a=\psi_p(v)= \psi^\pm(v)(p)
\prec 
+\psi^\pm(v)(p)=\psi^\pm(+v)(p)
=\psi_p(+v)=b$ and in the second case, $a=\psi_p(-v)=
\psi^\pm(-v)(p)=-\psi^\pm(v)(p)
\prec 
\psi^\pm(v)(p)
=\psi_p(v)=b$. So ${\diagcov}\subseteq {\prec}$. 

Observe that for $i\in \{1,\ldots,n\}$, we have $\psi_p(u)\in Dom(g_i)$ 
iff $u,x_i u\in\Delta_{\varepsilon}$ iff 
there exist $m\geq 0$, $v$ and $\sigma_1,\ldots,\sigma_m\in\{-1,0\}$ such that $u=\sigma_1 x_i^{(1)}\ldots\sigma_m x_i^{(m)}v$ and $(i,m,v)\in S$, or if 
there exist $m<0$, $v$ and $\sigma_1,\ldots,\sigma_{|m|}\in\{0,1\}$ such that $u=\sigma_1 x_i^{(-1)}\ldots\sigma_{|m|} x_i^{(m)}v$ and $ (i,|m|,v)\in S$. 
Given that $\psi_p (\sigma_1 x_i^{(1)}\ldots\sigma_m x_i^{(m)}v)=\psi^\pm(\sigma_1 x_i^{(1)}\ldots\sigma_m x_i^{(m)}v)(p)= 
(\psi^\pm(\sigma_1 x_i^{(1)}\ldots\sigma_m x_i^{(m)}) \circ \psi^\pm(v))(p)=\sigma_1 f_i^{(1)}\ldots$ $\sigma_m f_i^{(m)} \psi_p (v)=\sigma_1 f_i^{(1)}\ldots\sigma_m f_i^{(m)}v(f)(p)$ for $m \geq 0$, and given that we have 
$\psi_p (\sigma_1 x_i^{(-1)}\ldots\sigma_{|m|} x_i^{(m)}v)=\sigma_1 f_i^{(-1)}\ldots\sigma_{|m|} f_i^{(m)}v(f)(p)$ for $m<0$, we get that $\psi_p(u)\in Dom(g_i)$ iff 
$\psi_p (u)\in \Lambda_{f_i,m}^{v_{fp}}$ for some $(i,m,v)\in S$. So $Dom(g_i)=\bigcup\{\Lambda_{f_{i},m}^{v_{fp}}: (i,m,v)\in S\}$.

We will show that  $\psi_p(x_i^{(m)}u)=g_i^{[m]}(\psi_p(u))$ for all $i \in \{1,\ldots,n\}$, $m\in\mathbb{Z}$ and $u,x_i^{(m)}u\in\Delta_{\varepsilon}$; we will provide the details for $m>0$, as the other case is similar.  
Observe first that if $x_i^{(m)}u\in\Delta_{\varepsilon}$, then $x_i^{(m)}u\in\Delta_{x_i,j}^{v}$ for some $(i,j,v)\in S$ with $j\geq m$, so $u=\sigma_{m+1}x_i^{(m+1)}\ldots\sigma_j x_i^{(j)}v$.
Consequently, for all $\sigma_1,\ldots,\sigma_{m}\in\{-1,0\}$, we have $\sigma_{1}f_i^{(1)}\ldots\sigma_{m}f_i^{(m)}(u(f)(p))=\sigma_{1}f_i^{(1)}\ldots\sigma_{j}f_i^{(j)}(v(f)(p))$; hence $\Lambda_{f_i,m}^{u_{fp}}\subseteq\Lambda_{f_i,j}^{v_{fp}}\subseteq Dom(g_i)$. Therefore, $g_{i}$ coincides with $f_i$ on $\Lambda_{f_{i},m}^{u_{fp}}$ and, by Lemma \ref{l: Delta subf,m,a}, we get $g_{i}^{[m]}( u(f)(p))=f_{i}^{(m)}(u(f)(p))$. So we have $\psi_p(x_i^{(m)}u)=(x_i^{(m)}u)(f)(p)=f_{i}^{(m)}(u(f)(p))=g_{i}^{[m]}( u(f)(p))=g_{i}^{[m]}(\psi_p(u))$. 

Therefore $\psi_p$ is a compatible surjection for $\varepsilon$. Finally, the  equation fails in $\psi_p$, since for all $j \in\{1,\ldots,k\}$, we have $\psi_p(w_j)=w_j(f)(p)<p=\psi_p(1)$. 
 \end{proof}
 
 \begin{theorem}\label{t: comsur to diagram} 
 If an equation fails in a compatible surjection, then it fails in a diagram based on a chain of size up to $|\Delta_\varepsilon|$.

 \end{theorem}
 \begin{proof}
     Suppose an equation $\varepsilon$ in intentional form $1\leq w_1\vee\ldots\vee w_k$ with variables $x_1,\ldots ,x_n$ fails in a compatible surjection $\varphi:\Delta_{\varepsilon}\rightarrow\mathbb{N}_q$; i.e., $\varphi(w_1), \ldots, \varphi(w_k)<\varphi(1)$. Also, recall the definition of the diagram  $\m D_{\varepsilon, \varphi}:= (\mathbb{N}_q, {\leq}, {\diagcov}, g_1, \ldots, g_k)$ and set $\m \Delta:=(\mathbb{N}_q, {\leq}, {\diagcov})$  and $1_\varphi:=\varphi (1)\in \mathbb{N}_q$.
      We denote by $\hat{\varphi}: \m {Ti}\rightarrow \m{Pf( \Delta)}$
     the intentional homomorphism extending the assignment $\hat{\varphi}(x_i)=g_i$ for all $1 \leq i\leq n$. We will show that $\varepsilon$ fails on  $\m D_{\varepsilon, \varphi}$ by verifying that for all $j\in \{1,\ldots,k\}$, we have
     $\hat{\varphi}(w_j)(1_\varphi)=\varphi(w_j)$, which then implies     
     $\hat{\varphi}(w_j)(1_\varphi)=\varphi(w_j)<\varphi(1)=\hat{\varphi}(1)(1_\varphi)$.
     More generally, we will show that for all $u\in FS$, $\hat{\varphi}(u)(1_{\varphi})=\varphi(u)$, by using induction on $FS$. 
     
     If $u=1$, then $\hat{\varphi}(1)(1_{\varphi})=id_\Delta(1_{\varphi}) =1_{\varphi}=\varphi(1)$.
     If $\hat{\varphi}(u)(1_{\varphi})=\varphi(u)$, then using the fact that $\hat{\varphi}$ is an intentional homomorphism extending the assignment $\hat{\varphi}(x_i)=g_i$ for all $1 \leq i\leq n$, and that $\varphi$ is a compatible surjection, we get that
     $\hat{\varphi}(x_i^{(m)}u)(1_{\varphi})=
     g_i^{[m]}\hat{\varphi}(u)(1_{\varphi})=g_i^{[m]}(\varphi(u))=\varphi(x_i^{[m]}u)$,
      for all $i\in\{1,\ldots,n\}$ and $m\in\mathbb{Z}$ with $x_i^{(m)}u\in FS$. 
 \end{proof}

Revisiting Example~\ref{e: diagram}, we give the final diagram for the equation $1 \leq x^{\ell}x$ and its failure $f^{\ell}f(7)=4<7$ exhibited by the function $f$ of Figure \ref{f:building D}. 
Given that
\begin{equation*}
\begin{aligned}[t]
f^{\ell}f(7)&=4\\
 -f^{\ell}f(7)&=3 \\
     id(7) & = 7 
\end{aligned}
\qquad
\begin{aligned}[t]
ff^{\ell}f(7)&=5\\
f{-}f^{\ell}f(7)&=2\\
f(7)&=5
\end{aligned}
\end{equation*}
we have 
\begin{align*}
\Lambda_{f,0}^{7}=\{7\}&\text{ and }\Delta_{f,0}^{7}=\{ 5,7\}\\
\Lambda_{f,1}^{f(7)}=\{3,4\}&\text{ and }\Delta_{f,1}^{f(7)}=\{ 2,3,4,5\}
\end{align*}
Therefore, the resulting diagram is $\mathbf{D}=( \{2,3,4,5,7\},\leq,\diagcov,g)$ where $\leq$ is inherited from $\mathbb{Z}$, ${\diagcov} = \{(3,4)\}$  and $g=\{ (3,2), (4,5), (7,5) \}$. The equation $1\leq x^\ell x$ fails in $\mathbf{D}$ since $g^{\ell}g(7)=4<7$.

\section{From a diagram to  \mbox{$\mathbf{F(\mathbb{Z})}$}}\label{s: F(Z)}

In this section, we show that if an equation fails in some diagram, then it also fails in $\m F (\mathbb{Z})$.

\medskip

We start by revisiting the equation we considered in the last section, but with a different failure. Note that $1\leq x^{\ell}x$ fails in the  diagram $(\{3,4,5,6,7,8,9\},\leq,\prec,g)$ given in Figure \ref{f:building f from D}, since $g^{[\ell]}g(7)=5<7$. We would like to build a function $f$ on $\mathbf{F(\mathbb{Z})}$ that will exhibit the failure of the equation. This can be done by extending $g$ to a function $f$ and also making sure that $g^{[\ell]}(7)=f^\ell(7)$.
More generally, we want the iterated inverses of $f$ to agree with those of $g$, at least on the points involved in the failure.
We define the function $f:\mathbb{Z}\rightarrow\mathbb{Z}$ by:

  \[
f(n)=\begin{cases}
n & \text{if }n\notin\{3,4,5,6,7,8,9\} \\
9 & \text{if } n\in \{8,9\}\\
6 & \text{if } n \in \{5,6,7\}\\
3 & \text{if } n \in \{4,3\}
\end{cases}
\]
The function $f$ and its dual residual $f^{\ell}$ are displayed in Figure \ref{f:building f from D}. If we evaluate the equation on $f$ and at the element $7$, we obtain $f^{\ell}f(7)=5<7$ and therefore $\mathbf{F(\mathbb{Z})}\nvDash 1\leq x^{\ell}x$.
Observe that this is only one of many possible ways to extend $g$ to a function $f$.

  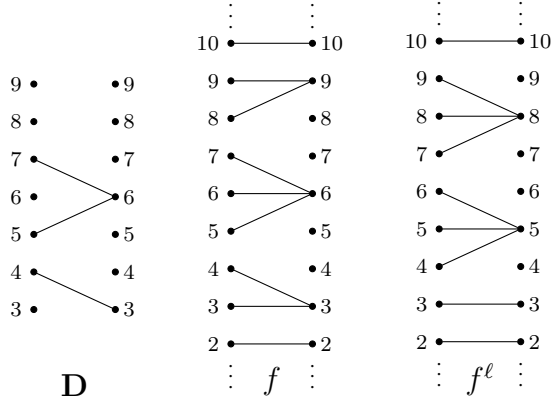
\begin{figure}[ht]\def\eq{=}.  
  \begin{center}
{\scriptsize
\begin{tikzpicture}
[scale=0.5]
\node[fill,draw,circle,scale=0.3,left](5) at (0,5){};
\node[left](5.1) at (-0.2,5){$9$};
\node[fill,draw,circle,scale=0.3,right](5r) at (2,5){};
\node[right](5.1) at (2.1,5){$9$};
\node[fill,draw,circle,scale=0.3,left](4) at (0,4){};
\node[left](4.1) at (-0.2,4){$8$};
\node[fill,draw,circle,scale=0.3,right](4r) at (2,4){};
\node[right](4.1) at (2.1,4){$8$};
\node[fill,draw,circle,scale=0.3,left](3) at (0,3){};
\node[left](3.1) at (-0.2,3){$7$};
\node[fill,draw,circle,scale=0.3,right](3r) at (2,3){};
\node[right](3.1) at (2.1,3){$7$};
\node[fill,draw,circle,scale=0.3,left](2) at (0,2){};
\node[left](2.1) at (-0.2,2){$6$};
\node[fill,draw,circle,scale=0.3,right](2r) at (2,2){};
\node[right](2.1) at (2.1,2){$6$};
\node[fill,draw,circle,scale=0.3,left](1) at (0,1){};
\node[left](1.1) at (-0.2,1){$5$};
\node[fill,draw,circle,scale=0.3,right](1r) at (2,1){};
\node[right](2.1) at (2.1,1){$5$};
\node[fill,draw,circle,scale=0.3,left](0) at (0,0){};
\node[left](0.1) at (-0.2,0){$4$};
\node[fill,draw,circle,scale=0.3,right](0r) at (2,0){};
\node[right](0.1) at (2.1,0){$4$};
\node[fill,draw,circle,scale=0.3,left](-1) at (0,-1){};
\node[left](0.1) at (-0.2,-1){$3$};
\node[fill,draw,circle,scale=0.3,right](-1r) at (2,-1){};
\node[right](0.1) at (2.1,-1){$3$};
\node at (1,-3){\normalsize$\mathbf{D}$};
\draw[-](3)--(2r);
\draw[-](1)--(2r);
\draw[-](0)--(-1r);
\end{tikzpicture}
\qquad
\begin{tikzpicture}
[scale=0.5]

\node[fill,draw,circle,scale=0.3,left](6) at (0,6){};
\node[left](6.1) at (-0.2,6){$10$};
\node[fill,draw,circle,scale=0.3,right](6r) at (2,6){};
\node[right](6.1) at (2.1,6){$10$};
\node[fill,draw,circle,scale=0.3,left](5) at (0,5){};
\node[left](5.1) at (-0.2,5){$9$};
\node[fill,draw,circle,scale=0.3,right](5r) at (2,5){};
\node[right](5.1) at (2.1,5){$9$};
\node[fill,draw,circle,scale=0.3,left](4) at (0,4){};
\node[left](4.1) at (-0.2,4){$8$};
\node[fill,draw,circle,scale=0.3,right](4r) at (2,4){};
\node[right](4.1) at (2.1,4){$8$};
\node[fill,draw,circle,scale=0.3,left](3) at (0,3){};
\node[left](3.1) at (-0.2,3){$7$};
\node[fill,draw,circle,scale=0.3,right](3r) at (2,3){};
\node[right](3.1) at (2.1,3){$7$};
\node[fill,draw,circle,scale=0.3,left](2) at (0,2){};
\node[left](2.1) at (-0.2,2){$6$};
\node[fill,draw,circle,scale=0.3,right](2r) at (2,2){};
\node[right](2.1) at (2.1,2){$6$};
\node[fill,draw,circle,scale=0.3,left](1) at (0,1){};
\node[left](1.1) at (-0.2,1){$5$};
\node[fill,draw,circle,scale=0.3,right](1r) at (2,1){};
\node[right](2.1) at (2.1,1){$5$};
\node[fill,draw,circle,scale=0.3,left](0) at (0,0){};
\node[left](0.1) at (-0.2,0){$4$};
\node[fill,draw,circle,scale=0.3,right](0r) at (2,0){};
\node[right](0.1) at (2.1,0){$4$};
\node[fill,draw,circle,scale=0.3,left](-1) at (0,-1){};
\node[left](-1.1) at (-0.2,-1){$3$};
\node[fill,draw,circle,scale=0.3,right](-1r) at (2,-1){};
\node[right](-1.1) at (2.1,-1){$3$};
\node[fill,draw,circle,scale=0.3,left](-2) at (0,-2){};
\node[left](-1.1) at (-0.2,-2){$2$};
\node[fill,draw,circle,scale=0.3,right](-2r) at (2,-2){};
\node[right](-1.1) at (2.1,-2){$2$};

\node at (-2)[below=-1pt]{$\vdots$};
\node at (-2r)[below=-1pt]{$\vdots$};
\node at (6)[above=3pt]{$\vdots$};
\node at (6r)[above=3pt]{$\vdots$};
\node at (1,-3){\normalsize$f$};
\draw[-](6)--(6r);
\draw[-](5)--(5r);
\draw[-](4)--(5r);
\draw[-](3)--(2r);
\draw[-](2)--(2r);
\draw[-](1)--(2r);
\draw[-](0)--(-1r);
\draw[-](-1)--(-1r);
\draw[-](-2)--(-2r);
\end{tikzpicture}
\qquad
\begin{tikzpicture}
[scale=0.5]

\node[fill,draw,circle,scale=0.3,left](6) at (0,6){};
\node[left](6.1) at (-0.2,6){$10$};
\node[fill,draw,circle,scale=0.3,right](6r) at (2,6){};
\node[right](6.1) at (2.1,6){$10$};
\node[fill,draw,circle,scale=0.3,left](5) at (0,5){};
\node[left](5.1) at (-0.2,5){$9$};
\node[fill,draw,circle,scale=0.3,right](5r) at (2,5){};
\node[right](5.1) at (2.1,5){$9$};
\node[fill,draw,circle,scale=0.3,left](4) at (0,4){};
\node[left](4.1) at (-0.2,4){$8$};
\node[fill,draw,circle,scale=0.3,right](4r) at (2,4){};
\node[right](4.1) at (2.1,4){$8$};
\node[fill,draw,circle,scale=0.3,left](3) at (0,3){};
\node[left](3.1) at (-0.2,3){$7$};
\node[fill,draw,circle,scale=0.3,right](3r) at (2,3){};
\node[right](3.1) at (2.1,3){$7$};
\node[fill,draw,circle,scale=0.3,left](2) at (0,2){};
\node[left](2.1) at (-0.2,2){$6$};
\node[fill,draw,circle,scale=0.3,right](2r) at (2,2){};
\node[right](2.1) at (2.1,2){$6$};
\node[fill,draw,circle,scale=0.3,left](1) at (0,1){};
\node[left](1.1) at (-0.2,1){$5$};
\node[fill,draw,circle,scale=0.3,right](1r) at (2,1){};
\node[right](2.1) at (2.1,1){$5$};
\node[fill,draw,circle,scale=0.3,left](0) at (0,0){};
\node[left](0.1) at (-0.2,0){$4$};
\node[fill,draw,circle,scale=0.3,right](0r) at (2,0){};
\node[right](0.1) at (2.1,0){$4$};
\node[fill,draw,circle,scale=0.3,left](-1) at (0,-1){};
\node[left](-1.1) at (-0.2,-1){$3$};
\node[fill,draw,circle,scale=0.3,right](-1r) at (2,-1){};
\node[right](-1.1) at (2.1,-1){$3$};
\node[fill,draw,circle,scale=0.3,left](-2) at (0,-2){};
\node[left](-1.1) at (-0.2,-2){$2$};
\node[fill,draw,circle,scale=0.3,right](-2r) at (2,-2){};
\node[right](-1.1) at (2.1,-2){$2$};

\node at (-2)[below=-1pt]{$\vdots$};
\node at (-2r)[below=-1pt]{$\vdots$};
\node at (6)[above=3pt]{$\vdots$};
\node at (6r)[above=3pt]{$\vdots$};
\node at (1,-3){\normalsize$f^{\ell}$};
\draw[-](6)--(6r);
\draw[-](5)--(4r);
\draw[-](4)--(4r);
\draw[-](3)--(4r);
\draw[-](2)--(1r);
\draw[-](1)--(1r);
\draw[-](0)--(1r);
\draw[-](-1)--(-1r);
\draw[-](-2)--(-2r);
\end{tikzpicture}
}

\caption{Building $f$ from $D$}\label{f:building f from D}
\end{center}
\end{figure}  

Let $F_{\textup{fs}}(\mathbb{Z})$ be the subset of $F(\mathbb{Z})$ consisting of the functions of finite support (they are equal to the identity function, except for finitely many places). It is easy to see that this defines a subalgebra of  $\m F(\mathbb{Z})$, which we denote by $\m F_{\textup{fs}}(\mathbb{Z})$.

\begin{theorem}\label{l:from D to F(Z)}
Every equation in the language of $\ell$-pregroups that fails in a diagram also fails in $\m F_{\textup{fs}}(\mathbb{Z})$.
\end{theorem}
\begin{proof}
Suppose $\varepsilon$ is an equation in intentional form $1\leq w_{1}\vee\ldots\vee w_{k}$ 
over variables $x_1,\ldots,x_n$ and that
$\mathbf{D}=({\mathbb{N}}_{q},\leq_{\Delta},\diagcov,g_{1},\ldots,g_{n})$
is a finite diagram in which the equation fails (observe that any finite diagram can give rise to a diagram of this form for an appropriated $q\in\mathbb{Z}^+$).
For each $i=1,\ldots ,n$ consider the map $f_{i}:\mathbb{Z}\rightarrow\mathbb{Z}$
where for all $a\in\mathbb{Z}$, if $J_a=\{b\in Dom(g_i) \cup\{q\}:a\leq b\}$:
\[
f_{i}(a)=\begin{cases}
a & \text{if }a\notin{\mathbb{N}}_{q}  \\
g_{i}(\bigwedge J_a) & \text{otherwise}
\end{cases}
\]
 To check that $f_{i}$ is order-preserving, first recall that $g_{i}$ is an order-preserving partial function and suppose that
$a_{1},a_{2}\in\mathbb{Z}$ with $a_{1}\leq a_{2}$.  If $a_{1},a_{2}\notin{\mathbb{N}}_{q}$, then $f_{i}(a_{1})=a_{1}\leq a_{2}=f_{i}(a_{2})$.
If $a_{1}\in{\mathbb{N}}_{q}  $ and $a_{2}\notin{\mathbb{N}}_{q} $, then
$f_{i}(a_{1})\leq q< a_{2}=f_{i}(a_{2})$.
Similarly, if $a_{1}\notin{\mathbb{N}}_{q}  $ and $a_{2}\in{\mathbb{N}}_{q}  $,
$f_{i}(a_{1})=a_{1} < 1 \leq f_{i}(a_{2})$.
Finally, if $a_{1},a_{2}\in{\mathbb{N}}_{q}$, then $J_{a_{2}} \subseteq J_{a_{1}}$, so by the order preservation of $g_i$ we get $f_{i}(a_{1})=g_{i}(\bigwedge J_{a_{1}})\leq g_{i}(\bigwedge J_{a_{2}})=f_{i}(a_{2})$.

By Lemma \ref{l:1&2 in integral}, to show that $f_{i}\in F(\mathbb{Z})$ it suffices to show that  $f_i$ is residuated and dually residuated.  Given that $f_i$ is order-preserving, it suffices to check (1) and (2) of Lemma~\ref{l: bounded preimage}. Since all $f_i$'s have finite support, this will show that they are in $\m F_{\textup{fs}}(\mathbb{Z})$.

For (1), given $a\in f_i[\mathbb{Z}]$, we observe that $f_i^{-1}[ a]$ is a convex subset of $\mathbb{Z}$, since $b\leq c\leq d$ and $f_i(b)=f_i(d)=a$, imply  $a=f_i(b)\leq f_i(c) \leq f_i(d)=a$, by the order preservation of $f_i$. Also, since $f_i$ is the identity outside $\mathbb{N}_q$,  $f_i^{-1}[a]$ is finite; hence $f_i^{-1}[a]$ is a bounded closed interval of $\mathbb{Z}$.

 We now argue that for any function on $\mathbb{Z}$, condition (1) actually implies (2). For any $a \not \in f_i[\mathbb{Z}]$, since $f_i[\mathbb{Z}]$ is not empty, there is at least one element of $f_i[\mathbb{Z}]$ above $a$ or below $a$. Without loss of generality, we assume it is above, and we take $z$ to be the least element of $f_i[\mathbb{Z}]$ above $a$. If there is no element of $f_i[\mathbb{Z}]$ below $a$, then $f_i^{-1}[z]$ would be an initial segment of $\mathbb{Z}$, which contradicts (1). So there is a largest element of $f_i[\mathbb{Z}]$ below $a$, say $y$; so $a \in (y,z)$. By (1), we have $f_i^{-1}[y]=[b_y,c_y]$ and $f_i^{-1}[z]=[b_z,c_z]$ for some $b_y,c_y,b_z,c_z\in\mathbb{Z}$, and by order-preservation (and the fact that $y<z$) we get $c_y<b_z$; so, $a \in (f_i(c_y), f_i(b_z))$. Finally, we have $c_y \prec b_z$; otherwise there exists $c_y<d<b_z$, so $y<f_i(d)<z$, contradicting the choice of either $y$ or $z$. Hence, $f_i\in F_\textup{fs}(\mathbb{Z})$.

Now, since $\varepsilon$ fails in $\m D$, there exist $p\in {\mathbb{N}}_{q}$ and an intentional homomorphism $\varphi: \m {Ti} \rightarrow \m {Pf}(\m \Delta)$ satisfying $\varphi(1)(p) > \varphi( w_{1})(p), \ldots  , \varphi(w_{k})(p)$ and $\varphi(x_i)=g_i$ for all $1 \leq i\leq n$. In particular, if $v \in FS$, then $\varphi(v)(p)$ is defined; also, if ${x_i}^{(m)}v \in FS$, then $\varphi(v)(p)\in Dom(g_i^{[m]})$. We define a homomorphism $\psi: \m {Ti} \rightarrow \m F(\mathbb{Z})$ extending the assignment $\psi(x_i)=f_i$. We will prove by induction that $\varphi(v)(p)=\psi(v)(p)$, for all $v \in FS$; this will yield 
$\psi(1)(p) > \psi( w_{1})(p), \ldots  , \psi(w_{k})(p)$, i.e., that $\varepsilon$ fails in $\m F_{\textup{fs}}(\mathbb{Z})$.

We have $\varphi(1)(p)=id_\Delta(p)=p=id_\mathbb{Z}(p)= \psi(1)(p)$. If  $\varphi(v)(p)=\psi(v)(p)$ and ${x_i}^{(m)}v \in FS$ for some $m$ and $i$, then  
$\varphi({x_i}^{(m)}v)(p)=\varphi({x_i}^{(m)}) \varphi(v)(p)=
g_i^{[m]}(\psi(v)(p))=
f_i^{(m)}(\psi(v)(p))=
\psi({x_i}^{(m)}) \psi(v)(p)
=\psi({x_i}^{(m)}v)(p)$, where we used $g_i^{[m]}(\psi(v)(p))=
f_i^{(m)}(\psi(v)(p))$. More generally, we establish that, for all $m\in\mathbb{Z}$,  $a\in Dom(g_{i}^{[m]})$ implies $f_{i}^{(m)}(a)=g_{i}^{[m]}(a)$. We will  show this for $m\in\mathbb{N}$ by induction on $m$, as the proof for $m\in\mathbb{Z}^-$ is dual.
 
For $m=0$, if $a\in Dom(g_{i})$, then $f_{i}^{(0)}(a)=f_i(a)=g_i(a)=g_{i}^{[0]}(a)$.
Now suppose that the statement holds for $m = s$, i.e., that  $a\in Dom(g_{i}^{[s]})$  implies $f_{i}^{(s)}(a)=g_{i}^{[s]}(a)$; we will show that it also holds for $m=s+1$.
If $a\in Dom(g_{i}^{[s+1]})$, then there are $b,c\in Dom(g^{[s]}_{i})$, 
such that 
$b\diagcov c$
and $g_{i}^{[s]}(b)<a\leq g_{i}^{[s]}(c)$; hence $g_{i}^{[s+1]}(a)=c$. Using the induction hypothesis for $b$ and $c$, and that  $b\prec c$, we get $f_{i}^{(s)}(b)=g_{i}^{[s]}(b)<a\leq g_{i}^{[s]}(c)=f_{i}^{(s)}(c)$; so 
$f_{i}^{(s+1)}(a)=c=g_{i}^{[s+1]}(a)$, by Lemma~\ref{l: bounded preimage}.
\end{proof}

By combining Theorem~\ref{t: DLPG to comsur}, Theorem~\ref{t: comsur to diagram}, and Theorem~\ref{l:from D to F(Z)}, together with the fact that $\m F(\mathbb{Z})$ is a distributive lattice-ordered pregroup and $\m F_{\textup{fs}}(\mathbb{Z})$ is a subalgebra of $\mathbf{F(\mathbb{Z})}$, we obtain the following result.

\begin{corollary}\label{c: cycle}
An equation $\varepsilon$ holds in  $\mathsf{DLPG}$ iff it holds in all compatible surjections for $\varepsilon$  iff it holds in all finite diagrams for $\varepsilon$  iff it holds in $\m F_{\textup{fs}}(\mathbb{Z})$ iff it holds in $\mathbf{F(\mathbb{Z})}$.
\end{corollary}

\begin{corollary}
  The variety $\mathsf{DLPG}$ is generated by $\m F_{\textup{fs}}(\mathbb{Z})$, hence also by $\mathbf{F(\mathbb{Z})}$.
\end{corollary}

\begin{corollary}\label{c:decidability}
The equational theory of $\mathsf{DLPG}$ is decidable.
\end{corollary}

\begin{proof}
Recall that every equation can be algorithmically transformed into an equivalent equation $\varepsilon$ that is in disjunctive intentional form. 
By Corollary~\ref{c: cycle}, $\varepsilon$ fails in $\mathsf{DLPG}$ iff it fails in some compatible surjection.
Since $\Delta_\varepsilon$ is finite, there are only finitely many options for compatible surjections  $\varphi:\Delta_\varepsilon \ra \mathbb{N}_q$, $q \leq |\Delta_\varepsilon|$, where $\varepsilon$ fails. By searching for all these finitely many $\varphi$'s we obtain an algorithm that verifies in finite time whether the equation holds in $\mathsf{DLPG}$. 

We provide a very crude upper bound for the number of compatible surjections for $\varepsilon$ in terms of the length $\ell$ of $\varepsilon$. First, note that $|FS_\varepsilon| \leq \ell$, $|Var| \leq \ell$, where $Var$ is the set of variables in $\varepsilon$, and $M\leq \ell$, where $M$ is the absolute value of the highest order of a residual in $\varepsilon$. So, $|S_\varepsilon|\leq |Var|\cdot M \cdot |FS_\varepsilon|\leq \ell^3$. Also, for every $x \in Var$, $v \in FS_\varepsilon$ and $m \in \mathbb{Z}$ with $|m| \leq M$, we have $|\Delta_{x, m}^v|\leq M \cdot 2^M\leq \ell \cdot 2^\ell$. Hence, $|\Delta_\varepsilon|\leq |S_\varepsilon| \cdot \ell \cdot 2^\ell \leq \ell^3 \cdot \ell \cdot  2^\ell=2^\ell \ell^4$.
So, the number of compatible surjections for $\varepsilon$ is at most the number of surjections from $\Delta_\varepsilon$ to an initial segment of $\mathbb{Z}^+$, which is at most $|\Delta_\varepsilon|^{|\Delta_\varepsilon|} \leq (2^\ell \ell^4)^{2^\ell \ell^4}$
and which is doubly exponential in $\ell$.
\end{proof}

\section{Compatible preorders}

In this section we provide an alternative and equivalent way of viewing compatible surjections for a given equation: compatible total preorders. This also provides a way to compare compatible preorders in the context of distrbutive $\ell$-pregroups to compatible preorders in the setting of $\ell$-groups and thus to total right orders on groups; see \cite{CM} and \cite{CGMS}. 

First, note that given any onto map $\varphi:\Delta\rightarrow I$ from a set $\Delta$ to an initial segment $I$ of $\mathbb{Z}^+$, we can define the binary relation $\preorder_{\varphi}$ on $\Delta$ by: $u\preorder_{\varphi} v$ iff $\varphi(u)\leq_{\mathbb{Z}^+}\varphi(v)$, for $u,v\in\Delta$. It is easy to prove  that $\preorder_{\varphi}$ is a preorder (it is reflexive and transitive) that is also total. A preorder $\preorder$ is called \emph{total} if for all $a,b\in L$, $a\preorder b$ or $b\preorder a$. Also, for all $u,v\in\Delta$ we write  $u\triangleleft v$ if $u\preorder v$ and $v\not\preorder u$.

Conversely, if $\preorder$ is a preorder on $\Delta$, it is well known that the relation
${\equiv}=\{(a,b)\in\Delta^2:a\preorder b\text{ and }b\preorder a \}$ is an equivalence relation and that the relation $\preorder_{\equiv}$ on the quotient set $\Delta / {\equiv}$ is a partial order, where $[u] \preorder_{\equiv} [v]$ iff $u \preorder v$. Also, if $\preorder$ is a total preorder, then $\preorder_{\equiv}$ is a total order. If $\Delta$ is finite, then the chain $(\Delta / {\equiv}, \preorder_{\equiv})$ is isomorphic (by a unique isomorphism) to the initial segment $\m \mathbb{N}_q$ of  $\mathbb{Z}^+$, where $q$ is the size of the chain. We denote by $\varphi_\preorder: \Delta \ra \mathbb{N}_q$ the composition of the quotient map and the isomorphism of the chains. 

  Furthermore, if $\Delta$ is a finite set, the maps $\varphi \mapsto {\preorder_{\varphi}}$ and ${\preorder} \mapsto \varphi_\preorder$ are mutually inverse bijections between total preorders on $\Delta$ and surjections from $\Delta$ to initial segments of $\mathbb{Z}^+$. Given an equation $\varepsilon$ in intentional form, we will lift this bijection to compatible surjections for $\varepsilon$ and compatible preorders for $\varepsilon$, by translating the condition for a compatible surjection in terms of the corresponding total preorder.

  \medskip

A total preorder $\preorder$ on the set $\Delta_{\varepsilon}$ is said to be a \emph{compatible preorder} if it satisfies the following conditions:
\begin{itemize}
\item[(i)] For all $i \in \{1,\ldots, n\}$, if $u,v, x_iu,x_iv \in{\Delta_{\varepsilon}}$ and $u\preorder v$, then $x_i u\preorder x_i v$.
\item[(ii)] If $u,-u\in\Delta_{\varepsilon}$, then $-u\triangleleft  u$ and for all $v\in\Delta_{\varepsilon}$, $-u\triangleleft v \Rightarrow u\preorder v$. Also, if $u,+u\in\Delta_{\varepsilon}$, then $u\triangleleft +u$ and for all $v\in\Delta_{\varepsilon}$, $u\triangleleft v \Rightarrow +u\preorder v$.
\item[(iii$\ell$)] If $i\in \{1,\ldots, n\}$, $j\in\mathbb{Z}^+$, and $v, x_i^{(j)}v\in\Delta_{\varepsilon}$, then $-x_i^{(j)}v$, $x_i^{(j-1)}{-}x_i^{(j)}v$, $x_i^{(j-1)}x_i^{(j)}v \in\Delta_{\varepsilon}$ and
\[
x_i^{(j-1)}{-}x_i^{(j)}v  \triangleleft v\preorder  x_i^{(j-1)}x_i^{(j)}v.
\]
\item[(iii$r$)]  If $i\in \{1,\ldots, n\}$, $j\in\mathbb{Z}^-$, and $v, x_i^{(j)}v\in\Delta_{\varepsilon}$, then $+x_i^{(j)}v$, $x_i^{(j+1)}{+}x_i^{(j)}v$, $x_i^{(j+1)}x_i^{(j)}v\in\Delta_{\varepsilon}$ and
\[
x_i^{(j+1)}x_i^{(j)}v\preorder v\triangleleft x_i^{(j+1)}{+}x_i^{(j)}v.
\]
\end{itemize}

\begin{theorem}
Given an equation $\varepsilon$ in intentional  form, the maps $\varphi \mapsto {\preorder_{\varphi}}$ and ${\preorder} \mapsto \varphi_\preorder$ are mutually inverse bijections between  compatible surjections for $\varepsilon$ and  compatible preorders for $\varepsilon$.
\end{theorem}

\begin{proof}
Suppose $\varepsilon$ is an equation in the language of DLPG in the intentional form $1\leq w_{1}\vee\ldots\vee w_{k}$ over variables $x_1, \ldots, x_n$; recall that $\Delta_\varepsilon$ is a finite set. Given an onto map 
$\varphi:\Delta_{\varepsilon}\rightarrow I$, where $I$ is an initial segment of $\mathbb{Z}^+$,
we will verify that $\preorder_{\varphi}$ is a total preorder.
For all $u\in\Delta_{\varepsilon}$, we have $u\preorder_{\varphi} u$ since $\varphi(u)\leq_{\mathbb{Z}^+}\varphi(u)$. Also, for all $u,v,w\in \Delta_{\varepsilon}$, $u\preorder_{\varphi} v$ and $v\preorder_{\varphi} w$ implies $\varphi(u)\leq_{\mathbb{Z}^+}\varphi(v)$ and $\varphi(v)\leq_{\mathbb{Z}^+}\varphi(w)$, so $\varphi(u)\leq_{\mathbb{Z}^+}\varphi(w)$, i.e.,  $u\preorder_{\varphi} w$. Finally, for all $u,v\in\Delta_{\varepsilon}$, 
$\varphi(u)\leq_{\mathbb{Z}^+}\varphi(v)$ or $\varphi(v)\leq_{\mathbb{Z}^+}\varphi(u)$, so $u\preorder_{\varphi} v$ or $v\preorder_{\varphi} u$. Therefore, $\preorder_{\varphi}$ is a total preorder.

Conversely, if $\preorder$ is a preorder on $\Delta_\varepsilon$, then
${\equiv} = {\preorder} \cap {\preorder^{-1}}$ is an equivalence relation  on $\Delta_\varepsilon$ and  $\preorder_{\equiv}$ is a  total order on the quotient set $\Delta_\varepsilon / {\equiv}$:  for all $u,v\in\Delta_{\varepsilon}$, 
$u\preorder v$ or $v\preorder u$, so
$[u]\preorder_{\equiv} [v]$ or $[v]\preorder_{\equiv} [u]$. Since $\Delta_\varepsilon$ is finite, the chain $(\Delta_\varepsilon / {\equiv}, \preorder_{\equiv})$ is isomorphic to the initial segment $\mathbb{N}_q$ of $\mathbb{Z}^+$, where $q=|\Delta_\varepsilon / {\equiv}|$, by a unique isomorphism $i_\varepsilon$. Then the composition $\varphi_\preorder=i_\varepsilon \circ \varphi_\equiv: \Delta_{\varepsilon} \ra \mathbb{N}_q$ is a map onto an  initial segment of $\mathbb{Z}^+$.

Moreover, if $\varphi:\Delta_{\varepsilon}\rightarrow I$ is a map onto an initial segment of $\mathbb{Z}^+$, then it factors uniquely as $\varphi=i_\varepsilon \circ \varphi_{\equiv_{\varphi}}$, where $\equiv_{\varphi}$ is the kernel of $\varphi$, $\varphi_{\equiv_{\varphi}}: \Delta_{\varepsilon}\rightarrow \Delta_\varepsilon / {\equiv_\varphi}$ is the quotient map and $i_\varepsilon: \Delta_\varepsilon / {\equiv_\varphi} \ra I$ is the unique isomorphism of these two finite chains. So   $\varphi_{\preorder_{\varphi}}=i_\varepsilon \circ \varphi_{\equiv_{\varphi}}=\varphi$, because ${\equiv}_\varphi= {\preorder_{\varphi}} \cap {\preorder_{\varphi}^{-1}}$. 
Also, if $\preorder$ is a total preorder on $\Delta_\varepsilon$ and $u,v\in\Delta_{\varepsilon}$, then $u\preorder_{\varphi_{\preorder}}v$ iff $\varphi_{\preorder}(u)\leq \varphi_{\preorder}(v)$ iff 
$i_\varepsilon(\varphi_{\equiv}(u))\leq i_\varepsilon(\varphi_{\equiv}(v))$ iff
${\varphi_{\equiv}}(u)\leq{\varphi_{\equiv}}(v)$ iff $[u]\preorder_{\equiv}[v]$ iff $u\preorder v$. So ${\preorder}={\preorder_{\varphi_{\preorder}}}$. Therefore, the  maps $\varphi \mapsto {\preorder_{\varphi}}$ and ${\preorder} \mapsto \varphi_\preorder$ are mutually inverse bijections between  total preorders on $\Delta_\varepsilon$ and maps from $\Delta_\varepsilon$ onto initial segments of $\mathbb{Z}^+$.

We will now show that the conditions in the definition of a compatible surjection $\varphi$ can be translated directly to the corresponding compatible preorder ${\preorder}: = {\preorder_\varphi}$. 

(i) The condition that $g_i$ is order-preserving is equivalent to the demand that for $u,v,x_i u,x_i v\in\Delta_{\varepsilon}$,  $\varphi(u)\leq_{\mathbb{Z}^+}\varphi(v) \Rightarrow \varphi(x_i u)\leq_{\mathbb{Z}^+} \varphi(x_i v)$, which is equivalent to 
$u\preorder v \Rightarrow x_i u\preorder x_i v$. 

(ii) The condition that ${\diagcov} \subseteq {\prec}$ is equivalent to: $u,-u\in\Delta_{\varepsilon} \Rightarrow \varphi(-u)\prec  \varphi(u)$ and
$u,+u\in\Delta_{\varepsilon} \Rightarrow \varphi(u)\prec  \varphi(+u)$.
Since $I$ is a chain, for $a,b \in I$, $a \prec b$ iff ($a < b$ and for all $c \in I$, $a<c \Rightarrow b \leq c$); also recall that $\varphi$ is onto, so every $c \in I$ is of the form $\varphi(v)$ for some $v\in\Delta_{\varepsilon}$. Therefore, the first implication is equivalent to: if $u,-u\in\Delta_{\varepsilon}$, then
$\varphi(-u)<  \varphi(u)$ and for all $v\in\Delta_{\varepsilon}$, $\varphi(-u)< \varphi(v) \Rightarrow \varphi(u) \leq \varphi(v)$.
Moreover this is equivalent to: if $u,-u\in\Delta_{\varepsilon}$, then
$-u\triangleleft  u$ and for all $v\in\Delta_{\varepsilon}$, $-u\triangleleft v \Rightarrow u\preorder v$.
Likewise, the other implication is equivalent to: if $u,+u\in\Delta_{\varepsilon}$, then
$u\triangleleft  +u$ and for all $v\in\Delta_{\varepsilon}$, $u\triangleleft v \Rightarrow +u\preorder v$. 


(iii$\ell$) Note that by the construction of $\Delta_{\varepsilon}$, if $ x_i^{(m)}u\in\Delta_{\varepsilon}$ then we also have $-x_i^{(m)}u,x_i^{(m-1)}{-}x_i^{(m)}u,x_i^{(m-1)}x_i^{(m)}u\in\Delta_{\varepsilon}$. So, the condition $i\in \{1,\ldots,n\}$, $m\in\mathbb{Z}^+$, $u,  x_i^{(m)}u\in\Delta_{\varepsilon}$ imply $g_i^{[m]}(\varphi(u))=\varphi(x_i^{(m)}u)$ is equivalent to the implication: $i\in \{1,\ldots,n\}$, $m\in\mathbb{Z}^+$, $u,  x_i^{(m)}u\in\Delta_{\varepsilon}$ imply  $\varphi(x_i^{(m-1)}{-}x_i^{(m)}u)=g^{[m-1]}(\varphi(-x_i^{(m)}u))$, $g_i^{[m]}(\varphi(u))=\varphi(x_i^{(m)}u)$, $g^{[m-1]}\varphi(x_i^{(m)}u)= \varphi(x_i^{(m-1)}x_i^{(m)}u)$. This is in turn equivalent to the implication:  $i\in \{1,\ldots,n\}$, $m\in\mathbb{Z}^+$, $u, x_i^{(m)}u\in\Delta_{\varepsilon}$ imply $\varphi(x_i^{(m-1)}{-}x_i^{(m)}u)=g^{[m-1]}(\varphi(-x_i^{(m)}u)) <u\leq g^{[m-1]}\varphi(x_i^{(m)}u)= \varphi(x_i^{(m-1)}x_i^{(m)}u)$.
Finally, translating this to the language of compatible preorders, we obtain the statement: for all $i\in \{1,\ldots,n\}$, $m\in\mathbb{Z}^+$ and $u\in\Delta_{\varepsilon}$, if $ x_i^{(m)}u\in\Delta_{\varepsilon}$ then $-x_i^{(m)}u,x_i^{(m-1)}{-}x_i^{(m)}u,x_i^{(m-1)}x_i^{(m)}u\in\Delta_{\varepsilon}$ and $x_i^{(m-1)}{-}x_i^{(m)}u  \triangleleft u\preorder  x_i^{(m-1)}x_i^{(m)}u$.

(iiir) The proof is similar to (iii$\ell$) .
\end{proof}

We say that an equation $1\leq w_{1}\vee\ldots\vee w_{k}$ in intentional form \emph{fails in a compatible preorder} $\preorder$ if $w_1, \ldots, w_k \triangleleft 1$.


\begin{corollary}
An equation  fails in a compatible preorder iff it fails in the corresponding compatible surjection.
\end{corollary}

%

As a result, the algorithm described in section 3 can alternatively be realized by searching for all possible compatible preorders.

\begin{example}
  To check the validity $1\leq x^{\ell}x$ in $\mathsf{DLPG}$ we first consider the sets $ \Delta_{x,0}^{1}=\{1,x\}$ and 
  $  \Delta_{x,1}^{x}=\{x,x^{\ell}x,-x^{\ell}x,xx^{\ell}x,x{-}x^{\ell}x\}$. Their union is the set $\Delta_{\varepsilon}=\{1,x,x^{\ell}x,-x^{\ell}x,xx^{\ell}x,x{-}x^{\ell}x\}$.
  We observe that the preorder $\preorder$, where
  $ -x^{\ell}x \; \triangleleft \; x^{\ell}x \; \triangleleft \; 1, x{-}x^{\ell}x \; \triangleleft \;  x, xx^{\ell}x$,
  is a compatible preorder in which $1\leq x^{\ell}x$ fails because $x^{\ell}x \triangleleft 1$. 
   The associated compatible surjection is $\varphi:\Delta_{\varepsilon}\rightarrow\{1,2,3,4\}$ where:
$$    \varphi(x) =\varphi(xx^{\ell}x)=4 \quad
   \varphi(1) =\varphi(x{-}x^{\ell}x)=3\quad
   \varphi(x^{\ell}x) =2\quad
   \varphi(-x^{\ell}x) =1$$
  The resulting diagram is   $(\mathbb{N}_4,\leq,\diagcov,g)$, where $\leq$ is induced by $\mathbb{Z}$,  $\diagcov=\{(1,2)\}$ and $g=\{(3,4),(2,4),(1,3)\}$, and it is  
  shown in Figure~\ref{f:building D algorithm}. 
  \end{example}
  
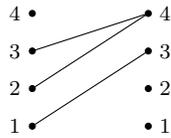
\begin{figure}[ht]\def\eq{=}
\begin{center}
{\scriptsize
\begin{tikzpicture}
[scale=0.5]
\node[fill,draw,circle,scale=0.3,left](2) at (0,2){};
\node[left](2.1) at (-0.2,2){$4$};
\node[fill,draw,circle,scale=0.3](2r) at (3,2){};
\node[right](2.1r) at (3.1,2){$4$};
\node[fill,draw,circle,scale=0.3,left](1) at (0,1){};
\node[left](1.1) at (-0.2,1){$3$};
\node[fill,draw,circle,scale=0.3](1r) at (3,1){};
\node[right](1.1r) at (3.1,1){$3$};
\node[fill,draw,circle,scale=0.3,left](0) at (0,0){};
\node[left](0.1) at (-0.2,0){$2$};
\node[fill,draw,circle,scale=0.3](0r) at (3,0){};
\node[right](0.1r) at (3.1,0){$2$};
\node[fill,draw,circle,scale=0.3,left](-1) at (0,-1){};
\node[left](-1.1) at (-0.2,-1){$1$};
\node[fill,draw,circle,scale=0.3](-1r) at (3,-1){};
\node[right](-1.1r) at (3.1,-1){$1$};
\draw[-](1)--(2r);
\draw[-](0)--(2r);
\draw[-](-1)--(1r);

\end{tikzpicture}
}
\caption{A diagram where  $1\leq x^{\ell}x$ fails}
\label{f:building D algorithm}
\end{center}
\end{figure}

\begin{example}
We now consider the equation $1\leq x^{\ell}\vee x$ and show that it holds in $\mathsf{DLPG}$. We first consider the sets
\begin{align*}
    \Delta_{x,1}^{1}&=\{1,x^{\ell},-x^{\ell},xx^{\ell},x{-}x^{\ell}\}\\
    \Delta_{x,0,}^{1}&=\{1,x\}\\
    \Delta_{\varepsilon}&=\{e,x,x^{\ell},{-}x^{\ell},xx^{\ell},x{-}x^{\ell}\}
\end{align*}
Moreover, we argue that there is no  compatible preorder $\preorder$ in which the equation fails.  Since $1\leq x^{\ell}\vee x$ fails, we have $x^{\ell}\triangleleft 1$ and $x\triangleleft 1$, and by (iii$\ell$) we get $1\preorder xx^{\ell}1$ (applied to $v=1$ and $m=1$). So, by transitivity $x\triangleleft xx^{\ell}$.
Also, by (i), $x^{\ell}\triangleleft 1$ yields $xx^{\ell} \preorder x$, which contradicts $x\triangleleft xx^{\ell}$.
\end{example}

\textbf{Acknowledgements:} We would like to thank Rick Ball and Peter Jipsen for stimulating discussions on distributive $\ell$-pregroups.

\end{document}